\documentclass[11pt,a4paper]{article}
\usepackage[T1]{fontenc}
\usepackage[utf8]{inputenc}
\usepackage{lmodern,microtype}
\usepackage[a4paper,margin=27mm]{geometry}
\usepackage{amsmath,amssymb,amsthm,mathtools,bm}
\usepackage{xcolor}
\usepackage{graphicx,booktabs,array,enumitem}\usepackage[section]{placeins}\usepackage{xurl}
\usepackage{tikz}
\usetikzlibrary{arrows.meta,positioning,decorations.pathreplacing,calc}
\usepackage{authblk}
\usepackage[numbers,sort&compress]{natbib}
\usepackage[hidelinks,bookmarksnumbered]{hyperref}
\usepackage{doi}
\usepackage{fancyhdr}
\allowdisplaybreaks[2]
\numberwithin{equation}{section}
\newtheorem{maintheorem}{Theorem}
\newtheorem{theorem}{Theorem}[section]
\newtheorem{lemma}[theorem]{Lemma}
\newtheorem{proposition}[theorem]{Proposition}
\newtheorem{corollary}[theorem]{Corollary}

\theoremstyle{definition}
\newtheorem{definition}[theorem]{Definition}
\newtheorem{hypothesis}[theorem]{Hypothesis}

\theoremstyle{remark}\newtheorem{remark}[theorem]{Remark}
\newcommand{\R}{\mathbb R}\newcommand{\Z}{\mathbb Z}\newcommand{\Sone}{\mathbb S^1}
\newcommand{\eps}{\varepsilon}\newcommand{\dd}{\mathrm d}\newcommand{\id}{\operatorname{Id}}
\newcommand{\wind}{\operatorname{wind}}\newcommand{\diag}{\operatorname{diag}}
\newcommand{\norm}[1]{\left\lVert#1\right\rVert}\newcommand{\abs}[1]{\left\lvert#1\right\rvert}
\newcommand{\Lip}{\operatorname{Lip}}

\title{Winding number forces universal synchronisation}

\author[1]{Jorge Buescu}
\author[2]{Henrique M. Oliveira\thanks{Corresponding author: \texttt{henrique.m.oliveira@tecnico.ulisboa.pt}.}}
\affil[1]{\small Departamento de Matem\'atica, Faculdade de Ci\^encias, and CEMS.UL--Center for Mathematical Studies, Universidade de Lisboa, Campo Grande, 1749-006 Lisbon, Portugal\\\texttt{jsbuescu@ciencias.ulisboa.pt}; ORCID: 0000-0001-5444-5089.}
\affil[2]{\small Departamento de Matem\'atica and Centro de An\'alise Matem\'atica, Geometria e Sistemas Din\^amicos, Instituto Superior T\'ecnico, Universidade de Lisboa, Av. Rovisco Pais 1, 1049-001 Lisbon, Portugal\\ORCID: 0000-0002-3346-4915.}
\date{}
\hypersetup{pdftitle={Winding number forces universal synchronisation},pdfauthor={Jorge Buescu and Henrique M. Oliveira}}

\begin{document}
\maketitle
\begin{abstract}
We prove that two completely different oscillators synchronise through arbitrarily weak reciprocal impulses. Any two planar oscillators with hyperbolic attracting limit cycles and equal natural periods lock stably one-to-one under translation impulses with generic directions and a suitable range of positive strength ratios, whatever their vector fields, limit cycles and transverse contraction rates. The mechanism is topological: the winding of the asymptotic phase gradient forces phase advances and delays, and the relative winding number of an effective impulse field gives lower bounds on the numbers of attracting and unstable locked states. Equal periods are not essential: locking persists under nonzero mismatches sufficiently small relative to the coupling strength, with persistence and loss governed by a shifted phase balance.

The analysis covers the full state-dependent impulsive dynamics, including transverse stability. Nearly simultaneous emissions may make the return discontinuous; under an explicit regularity condition the global return has a normally attracting Lipschitz invariant circle, with topologically Morse--Smale dynamics for nondegenerate phase balances away from simultaneous emission.
\end{abstract}


\section{Introduction}\label{sec:intro}

We consider two planar autonomous systems
\[
\dot x_1=F_1(x_1),\qquad
\dot x_2=F_2(x_2),\qquad
x_1,x_2\in\R^2,
\]
each with a hyperbolic attracting limit cycle, and ask when weak impulsive coupling locks them into a common periodic motion. Nothing relates $F_1$ to $F_2$: the cycles may differ in shape, in speed and in rate of transverse attraction, and $\norm{F_1-F_2}$ is not assumed small.

The coupling is reciprocal and triggered by the state. Each oscillator has a transverse section which it crosses once per revolution; when oscillator~1 crosses its section, the state of oscillator~2 jumps by an amount of order $\eps$, and conversely (Figure~\ref{fig:geometry}). Between impulses each oscillator follows its own equation. We do not postulate a phase model: locking is derived from the return map of the full impulsive system on a three-dimensional section. For any two such oscillators with equal natural periods, Theorem~\ref{thm:opencouplings} guarantees stable locking for generic directions of arbitrarily weak reciprocal translation impulses, when the impulse received by oscillator~1 is sufficiently weak relative to that received by oscillator~2.

\begin{figure}[htbp]
	\centering
	\includegraphics[width=.75\linewidth]{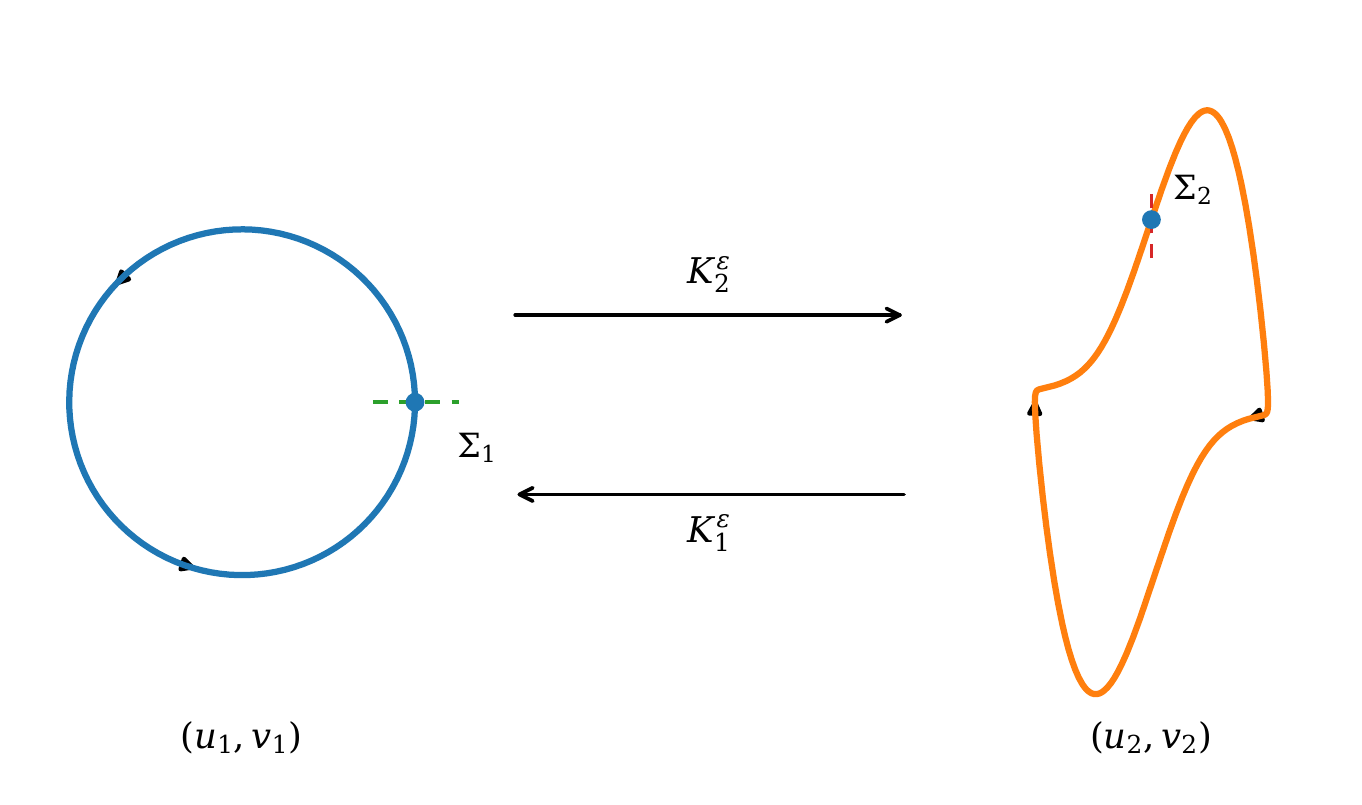}
	\caption{Reciprocal impulsive coupling of two attracting limit cycles $\Gamma_1$ and $\Gamma_2$: when oscillator~$i$ crosses its emission section $\Sigma_i$, the other oscillator $j$ is perturbed by $K_j^\eps$. Left: radial normal form, with $K_1^\eps(u_1,v_1)=(u_1-\beta\eps,v_1)$. Right: Van der Pol oscillator, with $K_2^\eps(u_2,v_2)=(u_2,v_2-\eps)$.}
	\label{fig:geometry}
\end{figure}

Most of the paper assumes equal natural periods, $T_1=T_2=T$. The relative asymptotic phase is then constant between impulses, and one pair of impulses displaces it by $\eps H$ to first order, where $H$ is the difference of the two phase responses. The zeros of $H$ are the candidates for locking. Where the order of the two emissions is fixed, the return map is smooth and the analysis is local. Near simultaneous emission the order can change and an impulse can carry the receiver across its own section, so one must first decide which crossings count as emissions; when the paired-event transition is regular, the global return map is Lipschitz. The global theorem assumes $H(0)\neq0$, which keeps every locked state away from simultaneous emission.

Equality of the periods is a codimension-one condition, and structural stability at fixed $\eps$ does not say how much mismatch is tolerated as $\eps\to0$. We therefore let the periods differ through a second parameter $\delta$, distinct from $\eps$, and then relate the two on the scale $|\delta|=O(\eps)$. The free drift $\Delta$ of the relative phase over one period adds to the impulsive displacement, so that for $\Delta=O(\eps)$ locking is governed by $\Delta+\eps H$.

\subsection{Relation to previous work}

Phase locking by weak interaction goes back to Adler's phase-balance equation \cite{Adler1946}; see \cite{Pikovsky2001,Strogatz2003} for broad accounts. The impulsive point of view comes from the study of clocks. Vassalo-Pereira analysed two interacting Andronov clocks \cite{AndronovVittKhaikin1966} \cite{VassaloPereira1982}, Nunes and Vassalo-Pereira a general interaction law \cite{NunesVassaloPereira1985}, and Abraham, later with Garfinkel, proposed a geometric theory of phase regulation by pulsatile forces for arbitrary oscillators \cite{Abraham1991,AbrahamGarfinkel2003}. We retain the idea common to these works, that short weak impacts act on the phase geometry of a stable oscillator. Their locking conclusions do not carry over: the zero-phase locking found in the early clock model neither matches Huygens' anti-phase observation nor holds in general for the exact impulsive dynamics; see \cite{OliveiraMelo2015}. Mechanical models of Huygens' clocks with specific couplings, and their in-phase and anti-phase regimes, are analysed in \cite{BennettSchatzRockwoodWiesenfeld2002,JovanovicKoshkin2012,KapitaniakEtAl2012}; impact models of clocks, which motivate the state jumps considered here, in \cite{OliveiraMelo2015,OliveiraPerestrelo2025}.

For smoothly coupled oscillators, Chicone \cite{Chicone1994} obtained periodic solutions near an invariant manifold of periodic orbits by Lyapunov--Schmidt reduction and Melnikov-type integrals, and Margheri and Martins \cite{MargheriMartins2010} studied generalised synchronisation of linearly coupled, not necessarily identical, time-periodic systems by invariant-manifold methods. Here the coupling is impulsive, the event times depend on the state, and the two vector fields are unrelated.

For pulse coupling, Mirollo and Strogatz proved almost-everywhere synchronisation of identical integrate-and-fire oscillators \cite{MirolloStrogatz1990}, and Kuramoto gave a phase description of pulse-coupled oscillators and excitable units \cite{Kuramoto1991Pulse}. Weak pulse coupling, phase-resetting methods and comparisons with diffusive coupling are developed in \cite{Izhikevich1999,GoelErmentrout2002,CanavierAchuthan2010,MauroySacreSepulchre2012}. Maps built from phase-resetting curves predict one-to-one locking in heterogeneous pairs, with higher-order corrections \cite{TalathiHwangMiliotisCarneyDitto2009} and with reciprocal coupling and conduction delays \cite{WoodmanCanavier2011}; see \cite{SrinivasCanavier2024} for later developments. Outside the weak-coupling regime, kicked oscillators can behave far less simply \cite{LinYoung2010}.

The phase of a point near an attracting cycle is its asymptotic phase, and the level sets of the phase are the isochrons \cite{Guckenheimer1975}; see \cite{Nakao2016} for phase reduction. Osinga and collaborators have studied their global geometry and computation: continuation of global isochrons \cite{OsingaMoehlis2010}, isochrons of the FitzHugh--Nagumo system \cite{LangfieldKrauskopfOsinga2014}, phase-resetting curves \cite{LangfieldKrauskopfOsinga2020}, and phase-resetting surfaces, phaseless points and critical level sets \cite{LeeBroderickKrauskopfOsinga2025,LeeBroderickKrauskopfOsinga2025Phaseless,LangfieldKrauskopfLeeOsinga2025}. That work describes how a perturbation changes the phase of one oscillator. We use the same geometry to decide when repeated weak impulses lock two distinct oscillators.

An impulse moves the receiver off its cycle, and the next impulse may arrive before the displacement has decayed. Phase and transverse displacement must therefore both be tracked, as in phase--amplitude response functions \cite{CastejonGuillamonHuguet2013}. For impulses at prescribed periodic times, P\'erez-Cervera, M-Seara and Huguet describe the phase dynamics on an invariant curve of the stroboscopic map \cite{PerezCerveraMSearaHuguet2019}, which persists by normal hyperbolicity \cite{HirschPughShub1977}. When impulses are triggered by section crossings the event times depend on the state, and this dependence enters the derivative of the return map \cite{ShirasakaKurebayashiNakao2017}.

We work directly with the Poincar\'e return of the coupled system. The two transverse fixed-point equations are solved in terms of the phase difference, which leaves the scalar equation $H=0$; stability is read from the full three-dimensional return, so that perturbations off the cycles are included. The same impulse can advance the phase at some points of a cycle and delay it at others \cite{Oprisan2013}, and we do not assume that $H$ has exactly two zeros. Instead we compare the rotation of an effective impulse field along the cycle with that of the tangent field: a nonzero relative winding number forces zeros of $H$, and simple zeros alternate between attracting and repelling.

\subsection{Organisation of the paper}

Section~\ref{sec:setting} presents the setting and the main results.
Section~\ref{sec:phasegeometry} introduces the phase geometry and
the Poincar\'e return, and Section~\ref{sec:returnproof} contains
the proofs of the local results.
Section~\ref{sec:global} develops the global analysis, including the master--slave case, while
Section~\ref{sec:nearbyperiods} addresses period mismatch and
structural stability.
Section~\ref{sec:examples} discusses the exceptional case $H(0)=0$,
presents three analytic examples and contains the numerical study,
and Section~\ref{sec:discussion} discusses the results.

\section{Setting and main results}\label{sec:setting}

\subsection{Oscillators, phases and impulses}
For $i=1,2$ let $F_i:D_i\to\R^2$ be a $C^2$ vector field on an open set $D_i\subset\R^2$, with local flow $\phi_t^i$. Between impulses the two oscillators evolve independently,
\begin{equation}\label{eq:fields}
	\dot x_1=F_1(x_1),\qquad \dot x_2=F_2(x_2).
\end{equation}

\begin{hypothesis}[Hyperbolic attracting cycles]\label{hyp:cycles}
	For $i=1,2$, $F_i$ has a periodic orbit $\Gamma_i$ of minimal period $T_i$ whose Floquet multipliers are $1$, along the orbit, and $\rho_i\in(0,1)$, transverse to it \cite{HirschSmaleDevaney2013}.
\end{hypothesis}

\begin{hypothesis}[Equal natural periods]\label{hyp:equalperiods}
	$T_1=T_2=:T$, and $\omega:=2\pi/T$.
\end{hypothesis}

Hypothesis~\ref{hyp:equalperiods} is in force except where a period mismatch is introduced, in Sections~\ref{sec:unequal} and~\ref{sec:nearbyperiods} and in the numerical example of Section~\ref{sec:numerics}; there it holds at $\delta=0$. No closeness of $F_1$ and $F_2$ is assumed. Fix $p_i\in\Gamma_i$, identify $\Sone$ with $\R/(2\pi\Z)$ and parametrise the cycle by phase,
\begin{equation}\label{eq:gammashift}
	\gamma_i(\theta):=\phi_{\theta/\omega}^i(p_i),
	\qquad
	\phi_t^i\bigl(\gamma_i(\theta)\bigr)=\gamma_i(\theta+\omega t).
\end{equation}

\begin{definition}[Asymptotic phase and isochron]\label{def:phase}
	A point $x$ of the basin $\mathcal A(\Gamma_i)=\{x\in D_i:\operatorname{dist}(\phi_t^i(x),\Gamma_i)\to0\text{ as }t\to+\infty\}$ has \emph{asymptotic phase} $\Theta_i(x)=\theta\in\Sone$ if $\|\phi_t^i(x)-\gamma_i(\theta+\omega t)\|\to0$ as $t\to+\infty$. The \emph{isochron} of phase $\theta$ is $\mathcal I_i(\theta):=\{x:\Theta_i(x)=\theta\}$.
\end{definition}

The phase $\Theta_i$ is $C^2$ on an open annular neighbourhood $U_i$ of $\Gamma_i$ whose closure is compact and contained in $\mathcal A(\Gamma_i)$, and it increases at the constant rate $\omega$ along trajectories (Theorem~\ref{thm:phaseclassical} and~\eqref{eq:phaseidentity}). Its gradient on the cycle, $Z_i(\theta):=\nabla\Theta_i(\gamma_i(\theta))$, is $C^1$ and normal to the isochron $\mathcal I_i(\theta)$, not in general to $\Gamma_i$.

\begin{definition}[Emission section]\label{def:emissionsection}
	Oscillator $i$ emits when it crosses an \emph{emission section} $\Sigma_i\subset U_i$: a connected $C^2$ embedded curve segment, transverse to $F_i$, with $p_i=\gamma_i(0)$ in its interior and $\Sigma_i\cap\Gamma_i=\{p_i\}$. Only crossings with a fixed orientation are counted.
\end{definition}

Fix open neighbourhoods $V_i$ of $\Gamma_i$ with $\overline{\Sigma_i}\subset V_i$ and $\overline{V_i}\subset U_i$ compact. The impulse, or \emph{kick}, received by oscillator $i$ is a map $K_i^\eps:V_i\to U_i$, $|\eps|<\eps_0$, such that $(x,\eps)\mapsto K_i^\eps(x)$ is $C^2$, $K_i^0=\id$ and each $K_i^\eps$ is a diffeomorphism onto its image; its first-order field is $B_i(x):=\partial_\eps K_i^\eps(x)|_{\eps=0}$. The subscript names the receiver: when $x_1^-\in\Sigma_1$, $x_2^+=K_2^\eps(x_2^-)$, and when $x_2^-\in\Sigma_2$, $x_1^+=K_1^\eps(x_1^-)$, where $x_i^\mp$ are the states immediately before and after the impulse. The emitter is unchanged at that instant, and between emissions~\eqref{eq:fields} applies. If both oscillators are on their sections at the same time, both impulses act on the pre-impulse states,
\begin{equation}\label{eq:simultaneous}
	(x_1^+,x_2^+)
	=\bigl(K_1^\eps(x_1^-),K_2^\eps(x_2^-)\bigr),
\end{equation}
unambiguously, since the two maps act on different factors. We write $\|f\|_{C^k(K)}$ for the supremum over a compact set $K$ of $f$ and its derivatives up to order $k$; all $O(\cdot)$ and $o(\cdot)$ refer to $\eps\to0$, uniformly on the compact sets indicated.

Theorems~\ref{thm:reciprocallocal}--\ref{thm:opencouplings} concern phase differences at which the two emissions occur in a fixed order. Each impulse is then received away from the receiver's own section, the emissions are exactly the transverse crossings of $\Sigma_1$ and $\Sigma_2$ for small $|\eps|$, and no further condition on the impulse maps is needed. Near simultaneous emission an impulse may carry the receiver across its own section; how emissions are then counted is specified in Section~\ref{sec:reciprocalglobal} and matters only for the global results, Theorem~\ref{thm:globalreciprocal} and its extensions in Section~\ref{sec:nearbyperiods}.

\subsection{The phase balance}
\begin{definition}[First-order phase response]\label{def:response}
	The \emph{first-order phase response} of oscillator $i$ to an impulse received at phase $\theta$ is
	\[
	h_i(\theta)
	:=\dd\Theta_{i,\gamma_i(\theta)}
	\bigl(B_i(\gamma_i(\theta))\bigr)
	=Z_i(\theta)\cdot B_i(\gamma_i(\theta)).
	\]
\end{definition}

Since $B_i$ and $Z_i$ are $C^1$, so is $h_i$, and Taylor's theorem gives, uniformly in $\theta$,
\begin{equation}\label{eq:phasejump}
	\Theta_i\bigl(K_i^\eps(\gamma_i(\theta))\bigr)
	=\theta+\eps h_i(\theta)+O(\eps^2):
\end{equation}
to first order the impulse advances the phase where $h_i>0$ and delays it where $h_i<0$.

The \emph{phase difference} of $(x_1,x_2)\in U_1\times U_2$ is $\psi(x_1,x_2):=\Theta_2(x_2)-\Theta_1(x_1)\pmod{2\pi}$; for equal periods it is constant between impulses, by~\eqref{eq:phaseidentity}. If oscillator~1 emits while both oscillators are on their cycles, oscillator~2 is at phase $\psi$ and receives its impulse there; when it next emits, oscillator~1 is at phase $-\psi$ to zeroth order and receives its impulse there. By~\eqref{eq:phasejump}, one return changes the phase difference by $\eps h_2(\psi)-\eps h_1(-\psi)+O(\eps^2)$, and at $\psi=0$ the simultaneous rule~\eqref{eq:simultaneous} gives the same expression.

\begin{definition}[Leading phase displacement]\label{def:H}
	Let
	\begin{equation}\label{eq:H}
		H(\psi):=h_2(\psi)-h_1(-\psi),
		\qquad \psi\in\Sone.
	\end{equation}
	A zero $\psi_*$ of $H$ is \emph{simple} if $H'(\psi_*)\neq0$; it is then \emph{downward}, or \emph{restoring}, if $H'(\psi_*)<0$, and \emph{upward}, or \emph{destabilising}, if $H'(\psi_*)>0$.
\end{definition}

\begin{remark}[Finiteness of simple zeros]\label{rem:finitezeros}
The terms of Definition~\ref{def:H} are used for the zeros of any $C^1$ function on $\Sone$. Simple zeros are isolated, so a $C^1$ function on the circle all of whose zeros are simple has finitely many.
\end{remark}

The function $H\in C^1(\Sone)$ is only the leading displacement of the phase difference: the return acts on a three-dimensional section and involves both transverse coordinates. Section~\ref{sec:returnproof} eliminates these before the scalar equation $H(\psi)=0$ appears, and Section~\ref{sec:reciprocalglobal} shows that the same $H$ governs the passage through simultaneous emission.

\begin{definition}[Locked state]\label{def:locking}
	A \emph{one-to-one locked state}, or simply \emph{locked state}, is a periodic trajectory of the impulsive system along which each oscillator emits exactly once per period. It is \emph{orbitally asymptotically stable} if every sufficiently close trajectory converges to it up to a translation of time; equivalently, if the corresponding fixed point of the return map on a transverse section is attracting \cite{HirschSmaleDevaney2013}.
\end{definition}

\subsection{Local locking}\label{sec:main}

Hypotheses~\ref{hyp:cycles} and~\ref{hyp:equalperiods} and the impulse law above are standing assumptions in Theorems~\ref{thm:reciprocallocal}--\ref{thm:globalreciprocal}.

\begin{maintheorem}[Reciprocal locking for distinct planar oscillators]\label{thm:reciprocallocal}
	Let $\psi_*\in(0,2\pi)$ be a simple zero of $H$. For every sufficiently small $\eps>0$ there is a locked state within $O(\eps)$ of $\Gamma_1\times\Gamma_2$, locally unique up to translation of time, whose phase difference at the emissions of oscillator~1 is $\psi_\eps=\psi_*+O(\eps)$. It is orbitally asymptotically stable if $\psi_*$ is downward, and unstable in the relative-phase direction if $\psi_*$ is upward.
\end{maintheorem}

\begin{corollary}[Multipliers, period and persistence]\label{cor:period}
	In the setting of Theorem~\ref{thm:reciprocallocal}, the fixed point of the return map corresponding to the locked state is hyperbolic, with multipliers
	\[
	\lambda_{\rm rel}=1+\eps H'(\psi_*)+o(\eps)
	\quad\text{in the relative-phase direction},
	\qquad
	\lambda_i=\rho_i+o(1)\quad(i=1,2).
	\]
	The locked state has period
	\begin{equation}\label{eq:period}
		\mathcal T_\eps
		=T-\frac{\eps}{\omega}h_1(-\psi_*)+O(\eps^2)
		=T-\frac{\eps}{\omega}h_2(\psi_*)+O(\eps^2),
	\end{equation}
	and, for each fixed sufficiently small $\eps>0$, it persists, together with its stability type and, since $\psi_*\neq0$, the order of the two emissions, under perturbations of the vector fields, emission sections and impulse maps that are sufficiently small in terms of $\eps$.
\end{corollary}

The two expressions in~\eqref{eq:period} agree because $H(\psi_*)=0$. If $H$ has exactly two zeros, both simple and neither at $\psi=0$, one is downward and the other upward: one locked state is attracting and the other unstable.

\subsection{Locking forced by winding}
Theorem~\ref{thm:reciprocallocal} requires a zero of $H$. A winding number argument forces such zeros to exist.

\begin{definition}[Winding number along a limit cycle]\label{def:winding}
	The \emph{winding number} along $\Gamma_i$ of a continuous nonvanishing vector field $V$ on $\Gamma_i$ is
	\[
	\wind_{\Gamma_i}(V)
	:=\deg\left(
	\theta\longmapsto
	\frac{V(\gamma_i(\theta))}{\norm{V(\gamma_i(\theta))}}
	\right),
	\]
	the topological degree of a map $\Sone\to\Sone$ \cite{Lee2013}; that is, the total change of a continuous argument of $V$ around the cycle, divided by $2\pi$.
\end{definition}

Since $\Gamma_i$ is a simple closed curve, its tangent field has $\wind_{\Gamma_i}(F_i)=\pm1$, according to orientation, by the Umlaufsatz \cite{Hopf1935}. To express the phase balance in the plane of oscillator~2 we define, along $\Gamma_2$, the \emph{effective field}
\begin{equation}\label{eq:effectivefield}
	W(\gamma_2(\psi))
	:=B_2(\gamma_2(\psi))
	-\frac{h_1(-\psi)}{\omega}
	F_2(\gamma_2(\psi)).
\end{equation}
Since $\dd\Theta_2(F_2)=\omega$ by~\eqref{eq:phasebasic},
\begin{equation}\label{eq:effectiveidentity}
	\dd\Theta_{2,\gamma_2(\psi)}
	\bigl(W(\gamma_2(\psi))\bigr)
	=Z_2(\psi)\cdot W(\gamma_2(\psi))
	=H(\psi).
\end{equation}
By~\eqref{eq:isochrontangent}, $H(\psi)=0$ exactly when $W$ is tangent at $\gamma_2(\psi)$ to the isochron $\mathcal I_2(\psi)$, and the zero is simple exactly when the tangency is nondegenerate (Figure~\ref{fig:theoremBlocal}). When $W$ does not vanish on $\Gamma_2$, its \emph{relative winding number} is
\begin{equation}\label{eq:windingnumber}
	m:=\wind_{\Gamma_2}(F_2)-\wind_{\Gamma_2}(W).
\end{equation}

\begin{figure}[htbp]
	\centering
	\includegraphics[width=.93\linewidth]{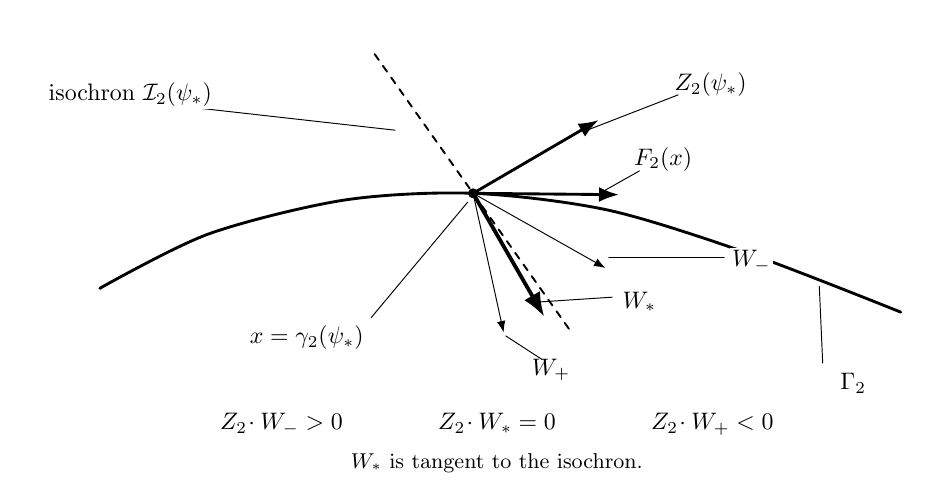}
	\caption{Geometry of Theorem~\ref{thm:reciprocalgeometric} at a zero $\psi_*$ of $H$, $x=\gamma_2(\psi_*)$. Since $H=Z_2\cdot W$, the effective field $W_*$ is tangent to the isochron $\mathcal I_2(\psi_*)$ (dashed), to which $Z_2(\psi_*)$ is normal; $F_2(x)$ is tangent to $\Gamma_2$. At a transverse contact the neighbouring directions $W_-$ and $W_+$ lie on opposite sides of the isochron, and $Z_2\cdot W$ changes sign.}
	\label{fig:theoremBlocal}
\end{figure}

\begin{maintheorem}[Geometric forcing of reciprocal locking]\label{thm:reciprocalgeometric}
	Suppose that $W$ does not vanish on $\Gamma_2$, that $m\neq0$, and that every zero of $H$ is simple. Then $H$ has at least $\abs m$ downward and at least $\abs m$ upward zeros, and each of them in $(0,2\pi)$ gives, for every sufficiently small $\eps>0$, a locked state as in Theorem~\ref{thm:reciprocallocal}. If $H(0)\neq0$, there are therefore at least $\abs m$ attracting locked states and at least $\abs m$ unstable ones.
\end{maintheorem}

\begin{remark}[Balance at simultaneous emission]\label{rem:simultaneous}
	If $H(0)=0$, the leading displacement of the phase difference vanishes exactly where the order of the emissions changes, and the two one-sided return maps near simultaneous emission must be examined separately; see Section~\ref{sec:exceptional}.
\end{remark}

\subsection{Translation impulses}
For constant translations $K_i^\eps(x)=x+\eps d_i$ with $d_i\in\R^2\setminus\{0\}$, the phase responses are $h_i(\theta)=Z_i(\theta)\cdot d_i$. We consider the reciprocal translations with strength ratio $\eta>0$,
\begin{equation}\label{eq:openreciprocal}
	K_1^\eps(x_1)=x_1+\eps\eta d_1,
	\qquad
	K_2^\eps(x_2)=x_2+\eps d_2.
\end{equation}

\begin{maintheorem}[Open sets of synchronising reciprocal translations]\label{thm:opencouplings}
	Fix the phase origins and $d_1\in\R^2\setminus\{0\}$. For every unit vector $d_2$ in an open, dense subset of the unit circle with complement of measure zero, there exist $\eta_0,\eps_0>0$ such that the translations~\eqref{eq:openreciprocal} have an orbitally asymptotically stable locked state for all $0<\eta<\eta_0$ and $0<\eps<\eps_0$. The conclusion persists under small perturbations of $d_1$, $d_2$ and $\eta$.
\end{maintheorem}

Both impulses in~\eqref{eq:openreciprocal} are nonzero, so the coupling is reciprocal and not its master--slave limit $\eta=0$, in which oscillator~2 is forced once per period by an unaffected oscillator~1 (Theorem~\ref{thm:strobe}). The proof continues a downward zero of $H$ from that limit, and hyperbolicity makes the conclusion open. Nothing is assumed about the shapes of the cycles or the distance between the two vector fields.

\subsection{Global dynamics}
Theorems~\ref{thm:reciprocallocal}--\ref{thm:opencouplings} are local in the phase difference. A global description must pass through simultaneous emission, where the order of the emissions changes and an impulse may carry the receiver across its own section. Section~\ref{sec:reciprocalglobal} fixes a convention under which each oscillator emits once in such an episode (Definition~\ref{def:singleemission}) and constructs a global return map $\mathcal R_\eps$ on a cylinder around the circle of phase differences. The construction needs the transition through the episode to be \emph{regular} (Definition~\ref{def:pairedregular}), a condition of Lipschitz closeness to the uncoupled transition, which in particular makes it continuous. It holds when each impulse map preserves the emission section of its receiver, and hence, for small $|\eps|$, the two sides of that section (Proposition~\ref{prop:sufficientregularity}), and may fail otherwise (Remark~\ref{rem:backwardkick}).

\begin{maintheorem}[Global reciprocal Morse--Smale locking]\label{thm:globalreciprocal}
	With emissions counted as in Definition~\ref{def:singleemission}, suppose that the paired-event transition is regular (Definition~\ref{def:pairedregular}), that $H(0)\neq0$, and that $H$ has at least one zero, all of them simple. Then, for every sufficiently small $\eps>0$, the return $\mathcal R_\eps$ has an invariant circle $\mathcal C_\eps$, a Lipschitz graph over the circle of phase differences, normally attracting in the metric sense, on which $q_\eps:=\mathcal R_\eps|_{\mathcal C_\eps}$ is topologically Morse--Smale. The periodic orbits of $q_\eps$ are precisely the locked states of Theorem~\ref{thm:reciprocallocal}, one near each zero of $H$, with the stability stated there.
\end{maintheorem}

Topologically Morse--Smale dynamics and metric normal attraction are defined in Section~\ref{sec:global} (Definition~\ref{def:topMS} and the paragraph before Theorem~\ref{thm:globalcircle}). The invariant circle exists without the assumptions on $H$ (Theorem~\ref{thm:globalcircle}).

\subsection{Unequal periods}\label{sec:unequal}
Finally we let the natural periods differ. Let the vector fields, emission sections and impulse maps depend $C^2$ on a second small parameter $\delta$, satisfy Hypotheses~\ref{hyp:cycles} and~\ref{hyp:equalperiods} at $\delta=0$ and the uniform phase hypotheses~\eqref{eq:phaseuniformdelta}--\eqref{eq:huniformdelta} of Section~\ref{sec:nearbyperiods}, and let $T_i^\delta$ be the period of the perturbed cycle of oscillator $i$. Over one natural period of oscillator~1 the free relative phase drifts by
\begin{equation}\label{eq:Deltadelta}
	\Delta_\delta
	:=2\pi\left(\frac{T_1^\delta}{T_2^\delta}-1\right).
\end{equation}
We consider paths $\delta(\eps)=O(\eps)$ along which
\begin{equation}\label{eq:nudrift}
	\Delta_{\delta(\eps)}=\eps\nu+o(\eps),
\end{equation}
and the shifted phase balance $H_\nu:=\nu+H$.

\begin{maintheorem}[Locking with period mismatch of order $\eps$]\label{thm:smallmismatch}
	Every simple zero $\psi_*\in(0,2\pi)$ of $H_\nu$ gives, for every sufficiently small $\eps>0$, a locally unique locked state with phase difference $\psi_\eps=\psi_*+o(1)$ and relative multiplier $\lambda_{\rm rel}=1+\eps H'(\psi_*)+o(\eps)$. It is orbitally asymptotically stable if $\psi_*$ is downward, and unstable in the relative-phase direction if $\psi_*$ is upward.
\end{maintheorem}

Corollary~\ref{cor:mismatchperiod} gives the locked period, and Theorem~\ref{thm:mismatchglobal} with Corollary~\ref{cor:mismatchMS} is the counterpart of Theorem~\ref{thm:globalreciprocal} for unequal periods.

\section{Phase geometry and the return map}\label{sec:phasegeometry}

\subsection{Regularity of the asymptotic phase}
The local isochrons of a hyperbolic attracting periodic orbit form an invariant foliation, with isochron map as regular as the vector field \cite[Theorem~1]{EngelKuehn2021}; see \cite{ChiconeLiu2004,Dumortier2006} for the planar theory, \cite{Guckenheimer1975} for the classical formulation and \cite{Nakao2016} for phase reduction. For $C^2$ vector fields this gives:

\begin{theorem}[Asymptotic phase]\label{thm:phaseclassical}
For $i=1,2$ there is an open annular neighbourhood $U_i$ of $\Gamma_i$, with $\overline{U_i}$ compact and contained in $\mathcal A(\Gamma_i)$, on which the asymptotic phase $\Theta_i:U_i\to\Sone$ is of class $C^2$, with $\Theta_i(\gamma_i(\theta))=\theta$ and $\dd\Theta_i\neq0$. Each level set $\mathcal I_i(\theta)\cap U_i$ is a $C^2$ embedded curve, and for $v\in\R^2$
	\begin{equation}\label{eq:isochrontangent}
		v\text{ is tangent to }\mathcal I_i(\theta)\text{ at }\gamma_i(\theta)
		\quad\Longleftrightarrow\quad
		\dd\Theta_{i,\gamma_i(\theta)}(v)=0.
	\end{equation}
\end{theorem}

Compactness of $\overline{U_i}$ serves only to make the estimates uniform where impulses are applied. The local isochron $\mathcal I_i(\theta)\cap U_i$ is the local stable manifold of the fixed point $\gamma_i(\theta)$ of the time-$T$ map $\phi_T^i$, which is the link with the invariant-foliation formulation of \cite{Guckenheimer1975,ChiconeLiu2004,Dumortier2006}.

By Definition~\ref{def:phase}, $\phi_s^i(x)$ has phase $\Theta_i(x)+\omega s$; hence, whenever $\phi_s^i(x)\in U_i$,
\begin{equation}\label{eq:phaseidentity}
	\Theta_i(\phi_s^i(x))
	=\Theta_i(x)+\omega s\pmod{2\pi}.
\end{equation}
Differentiating~\eqref{eq:gammashift} and~\eqref{eq:phaseidentity}, in any local real-valued lift of $\Theta_i$, gives
\begin{equation}\label{eq:phasebasic}
	\gamma_i'(\theta)=\frac{F_i(\gamma_i(\theta))}{\omega},
	\qquad
	\dd\Theta_{i,x}\bigl(F_i(x)\bigr)=\omega,
\end{equation}
and in particular
\begin{equation}\label{eq:normalisation}
	Z_i(\theta)\cdot F_i(\gamma_i(\theta))=\omega>0.
\end{equation}
Since $\Theta_i$ and $\gamma_i$ are $C^2$, $Z_i\in C^1(\Sone;\R^2)$. Differentiating $\dd\Theta_i(F_i)=\omega$ along the orbit shows that $\widehat Z_i(t):=Z_i(\omega t)$ is the periodic solution of the adjoint equation
\[
\dot{\widehat Z}_i(t)
=-DF_i(\gamma_i(\omega t))^{T}\widehat Z_i(t),
\qquad
\widehat Z_i(t)\cdot F_i(\gamma_i(\omega t))=\omega,
\]
the infinitesimal phase-response vector of phase-reduction theory \cite{Nakao2016}; since $DF_i\circ\gamma_i$ is $C^1$, $\widehat Z_i$ is in fact $C^2$.

\subsection{Sections, impulses and hitting times}\label{sec:sectionsimpulses}
A transverse section is locally a regular level set \cite{Lee2013}: every $y\in\Sigma_i$ has a neighbourhood $B\subset U_i$ and a $C^2$ function $g:B\to\R$ with $\Sigma_i\cap B=\{g=0\}$ and $\dd g_y(F_i(y))>0$, the sign being chosen so that recorded crossings go from the \emph{incoming side} $\{g<0\}$ to the \emph{outgoing side} $\{g>0\}$.

By Taylor's theorem on the compact set $\overline{V_i}$, $K_i^\eps(x)=x+\eps B_i(x)+R_i^K(x,\eps)$ with $\|R_i^K(\cdot,\eps)\|_{C^0(\overline{V_i})}=O(\eps^2)$ and $\|R_i^K(\cdot,\eps)\|_{C^1(\overline{V_i})}=o(\eps)$.

We use repeatedly the implicit function theorem in the following form \cite{HirschSmaleDevaney2013}: if a free trajectory starting at $x_0$ crosses a transverse section $\Sigma=\{g=0\}$ at time $\tau_0>0$, the hitting time is a $C^2$ function of the initial state near $x_0$, unique near $\tau_0$, because $\partial_t\,g(\phi_t(x))=\dd g(F)\neq0$ at the crossing.

\subsection{The Poincar\'e return and the exact phase identity}\label{sec:localreturn}
Between impulses the state $(x_1,x_2)$ evolves by the product flow $\Phi_t(x_1,x_2):=\bigl(\phi_t^1(x_1),\phi_t^2(x_2)\bigr)$. We sample it immediately before oscillator~1 emits, on the three-dimensional section $\mathcal S_1:=\Sigma_1\times V_2$. Fix a point of $\mathcal S_1$ from which oscillator~2 emits strictly after oscillator~1 and strictly before the next emission of oscillator~1. By the hitting-time result of Section~\ref{sec:sectionsimpulses} there are a neighbourhood $\mathcal V\subset\mathcal S_1$ of that point and $\eps_1\in(0,\eps_0)$ such that, for $z=(x_1,x_2)\in\mathcal V$ and $|\eps|<\eps_1$, this order of emissions persists, every impulse is applied at a point of the corresponding $V_i$, and the following construction is well defined. Apply the impulse to oscillator~2; flow until oscillator~2 first reaches $\Sigma_2$, at time $\tau_2(z,\eps)>0$; apply the impulse to oscillator~1; flow until oscillator~1 first reaches $\Sigma_1$, at time $\tau_1(z,\eps)>0$:
\begin{equation}\label{eq:Poincare}
	\begin{aligned}
		z^{(1)}&=\bigl(x_1,K_2^\eps(x_2)\bigr),&
		z^{(2)}&=\Phi_{\tau_2(z,\eps)}(z^{(1)}),\\
		z^{(3)}&=\bigl(K_1^\eps(z^{(2)}_1),z^{(2)}_2\bigr),&
		\mathcal P_\eps(z)&:=\Phi_{\tau_1(z,\eps)}(z^{(3)}),
	\end{aligned}
\end{equation}
all crossings being taken with the chosen orientation. Then $\mathcal P_\eps(z)\in\mathcal S_1$, and $\tau_1$, $\tau_2$ and the \emph{local Poincar\'e return} $\mathcal P_\eps$ are $C^2$ in $(z,\eps)$, after shrinking $\mathcal V$ if necessary (Figure~\ref{fig:returnsequence}).

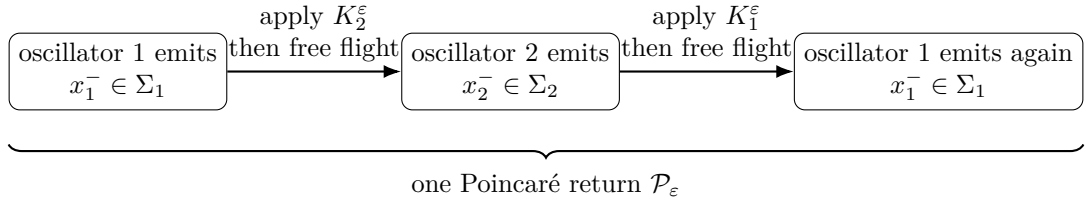
\begin{figure}[htbp]
	\centering
	\begin{tikzpicture}[>=Latex, node distance=2.3cm, every node/.style={font=\small}]
		\node[draw,rounded corners,align=center,minimum width=2.55cm,minimum height=.9cm] (e1a) {oscillator $1$ emits\\$x_1^-\in\Sigma_1$};
		\node[draw,rounded corners,align=center,right=of e1a,minimum width=2.55cm,minimum height=.9cm] (e2) {oscillator $2$ emits\\$x_2^-\in\Sigma_2$};
		\node[draw,rounded corners,align=center,right=of e2,minimum width=2.55cm,minimum height=.9cm] (e1b) {oscillator $1$ emits again\\$x_1^-\in\Sigma_1$};
		\draw[->,thick] (e1a) -- node[above,align=center] {apply $K_2^\eps$\\then free flight} (e2);
		\draw[->,thick] (e2) -- node[above,align=center] {apply $K_1^\eps$\\then free flight} (e1b);
		\draw[decorate,decoration={brace,mirror,amplitude=5pt},thick] ($(e1a.south west)+(0,-.45)$) -- node[below=7pt] {one Poincar\'e return $\mathcal P_\eps$} ($(e1b.south east)+(0,-.45)$);
	\end{tikzpicture}
	\caption{The local Poincar\'e return when the two emissions occur at distinct times. The state is sampled immediately before successive emissions of oscillator~$1$.}
	\label{fig:returnsequence}
\end{figure}

The equal-period analysis rests on the following identity.

\begin{lemma}[Exact phase identity]\label{lem:phaseidentity}
	Let $z=(x_1,x_2)\in\mathcal V$ and let $y_1:=z^{(2)}_1=\phi^1_{\tau_2(z,\eps)}(x_1)$ be the state of oscillator~1 when oscillator~2 emits. Then
	\begin{equation}\label{eq:exactphase}
		\psi\bigl(\mathcal P_\eps(z)\bigr)-\psi(z)
		=\bigl[\Theta_2(K_2^\eps(x_2))-\Theta_2(x_2)\bigr]
		-\bigl[\Theta_1(K_1^\eps(y_1))-\Theta_1(y_1)\bigr]
		\pmod{2\pi}.
	\end{equation}
	In particular $\psi\circ\mathcal P_0=\psi$ on all of $\mathcal V$, not only on the cycles.
\end{lemma}

\begin{proof}
	By~\eqref{eq:phaseidentity}, a free flight of duration $t$ increases $\Theta_1$ and $\Theta_2$ by $\omega t$, so $\psi$ is constant along the product flow. The return consists of two free flights and two impulses. The impulse on oscillator~2 changes $\Theta_2$ by the first bracket and leaves $\Theta_1$ unchanged; the impulse on oscillator~1 changes $\Theta_1$ by the second bracket and leaves $\Theta_2$ unchanged.
\end{proof}

On $\mathcal S_1$ the product of the two cycles is the circle $\mathcal C_0:=\{(p_1,\gamma_2(\psi)):\psi\in\Sone\}$, on which oscillator~1 is at its emission point and oscillator~2 has phase $\psi$; at $\eps=0$ every point of $\mathcal C_0$ is fixed by the return.

We need $C^2$ coordinates near this circle. Let $J\subset(0,2\pi)$ be a compact interval and $\mathcal C_0(J):=\{(p_1,\gamma_2(\psi)):\psi\in J\}$. After shrinking the annular neighbourhoods if necessary, choose $C^2$ functions $r_i$ with $\Gamma_i=\{r_i=0\}$ and $\dd r_i\neq0$ on $\Gamma_i$. On $\Gamma_i$ the kernel of $\dd r_i$ is spanned by $F_i$, whereas $\dd\Theta_i(F_i)=\omega\neq0$; hence $(\Theta_i,r_i)$ are $C^2$ coordinates near every compact arc of $\Gamma_i$. Since $\Sigma_1$ is transverse to $F_1$ at $p_1$, the restriction of $r_1$ to $\Sigma_1$ is a $C^2$ coordinate near $p_1$. Therefore $\psi=\Theta_2-\Theta_1$, $a_1=r_1(x_1)$ and $a_2=r_2(x_2)$ are $C^2$ coordinates on a neighbourhood of $\mathcal C_0(J)$ in $\mathcal S_1$ \cite{Lee2013}, which for some $r_*>0$ they identify with
\begin{equation}\label{eq:tubecoordinates}
	J\times B_{r_*}(0)\subset\R\times\R^2,
	\qquad
	\mathcal C_0(J)=\{a_1=a_2=0\},
\end{equation}
where $B_{r_*}(0)=\{a\in\R^2:\|a\|<r_*\}$. The coordinates $a_1$ and $a_2$ are the transverse displacements of the two oscillators from their cycles.

\section{The return map and the proofs of Theorems~\ref{thm:reciprocallocal}--\ref{thm:opencouplings}}\label{sec:returnproof}

\subsection{Uniform first-order expansion of the phase change}
Since $\Theta_i\in C^2(U_i)$ by Theorem~\ref{thm:phaseclassical} and $(x,\eps)\mapsto K_i^\eps(x)$ is $C^2$ on $\overline{V_i}\times[-\eps_0,\eps_0]$, Taylor's theorem in $\eps$, applied to a local real-valued lift of $\Theta_i$ on a simply connected neighbourhood of a compact arc of the cycle, gives
\begin{equation}\label{eq:C2Taylor}
	\Theta_i(K_i^\eps(x))-\Theta_i(x)
	=\eps\,\dd\Theta_{i,x}(B_i(x))+R_i(x,\eps),
\end{equation}
where, on every compact set $K\subset V_i$,
\[
\|R_i(\cdot,\eps)\|_{C^0(K)}\le C_K\eps^2,
\qquad
\|R_i(\cdot,\eps)\|_{C^1(K)}=o(\eps).
\]
The first bound comes from the second derivative in $\eps$. The second holds because $D_xR_i$ is $C^1$ in $\eps$ and vanishes to first order at $\eps=0$.

\subsection{Expansion of the Poincar\'e return}
Let $J\subset(0,2\pi)$ be a compact interval and $d_J:=\min_{\psi\in J}\min\{\psi,2\pi-\psi\}>0$. In the uncoupled system, if oscillator~1 emits at phase difference $\psi\in J$, oscillator~2 emits after time $(2\pi-\psi)/\omega$ and oscillator~1 again after a further time $\psi/\omega$; both gaps are at least $d_J/\omega$. With transversality of the sections and compactness of $J$, this implies that the order of emissions persists uniformly for small transverse displacements and small $|\eps|$.

In the coordinates~\eqref{eq:tubecoordinates} the uncoupled transverse return has derivative $D:=\diag(\rho_1,\rho_2)$ at $a=0$. For $r_*>0$ let $\mathcal N_{J,r_*}:=J\times\overline{B_{r_*}(0)}$ be the closed coordinate tube. After decreasing $r_*$ and $\eps_0$ if necessary, it lies in the coordinate domain and the return~\eqref{eq:Poincare} is defined on a neighbourhood of it for every $|\eps|<\eps_0$.

\begin{proposition}[Uniform return-map expansion]\label{prop:returnexp}
	There exist $r_*,\eps_0>0$ such that, for $(\psi,a)\in\mathcal N_{J,r_*}$ and $|\eps|<\eps_0$, the return $\mathcal P_\eps$ is $C^2$ and has the form
	\begin{align}
		\psi'&=\psi+\eps H(\psi)+\mathcal E(\psi,a,\eps),\label{eq:recphase}\\
		a'&=Da+\eps b(\psi)+Q(\psi,a,\eps),\nonumber
	\end{align}
	with $b\in C^1(J;\R^2)$ and, uniformly on $\mathcal N_{J,r_*}$,
	\begin{align}
		|\mathcal E(\psi,a,\eps)|
		&\le C\bigl(\eps^2+|\eps|\,\|a\|\bigr),\label{eq:Evalue}\\
		\|D_a\mathcal E(\psi,a,\eps)\|
		&\le C|\eps|,\label{eq:Eaderivative}\\
		\|Q(\psi,a,\eps)\|
		&\le C\bigl(\|a\|^2+|\eps|\,\|a\|+\eps^2\bigr),\label{eq:Qvalue}\\
		\|D_aQ(\psi,a,\eps)\|+\|\partial_\psi Q(\psi,a,\eps)\|
		&\le C\bigl(\|a\|+|\eps|\bigr).\label{eq:Qderivatives}
	\end{align}
\end{proposition}

\begin{proof}
	The return~\eqref{eq:Poincare} is a composition of $C^2$ free-flight maps, impulse maps and hitting-time functions. In the coordinates~\eqref{eq:tubecoordinates} it is therefore $C^2$ jointly in $(\psi,a,\eps)$ on a neighbourhood of the compact tube $\mathcal N_{J,r_*}$, and compactness makes all constants below uniform.
	
	\emph{Phase component.} By Lemma~\ref{lem:phaseidentity}, $\psi'-\psi$ is the difference of the two impulsive phase changes in~\eqref{eq:exactphase}; in particular it vanishes identically at $\eps=0$, whatever the transverse coordinates. Let $G(\psi,a)$ be the derivative of $\psi'$ with respect to $\eps$ at $\eps=0$. Since the return is $C^2$, $G$ is $C^1$ and, as in the preceding subsection,
	\[
	\psi'-\psi=\eps\,G(\psi,a)+\widetilde{\mathcal E}(\psi,a,\eps),
	\qquad
	|\widetilde{\mathcal E}|\le C\eps^2,
	\qquad
	\|\widetilde{\mathcal E}(\cdot,\cdot,\eps)\|_{C^1}=o(\eps).
	\]
	By~\eqref{eq:C2Taylor}, $G$ is the difference of the values of $\dd\Theta_i(B_i)$ at the two points where the impulses are received at $\eps=0$. For $a=0$ these points are $\gamma_2(\psi)$ and $\gamma_1(-\psi)$, so $G(\psi,0)=h_2(\psi)-h_1(-\psi)=H(\psi)$. Hence~\eqref{eq:recphase} holds with
	\[
	\mathcal E(\psi,a,\eps)
	=\eps\bigl[G(\psi,a)-G(\psi,0)\bigr]+\widetilde{\mathcal E}(\psi,a,\eps),
	\]
	and~\eqref{eq:Evalue} and~\eqref{eq:Eaderivative} follow because $G$ is $C^1$.
	
	\emph{Transverse components.} At $\eps=0$ the transverse return has derivative $D$ at $a=0$: over one period the two transverse variations are multiplied by $\rho_1$ and $\rho_2$. The impulses contribute a term of order $\eps$, written $\eps b(\psi)$, and Taylor's theorem in $(a,\eps)$ on the tube gives~\eqref{eq:Qvalue} and~\eqref{eq:Qderivatives}.
\end{proof}

Within $O(\eps)$ of the circle $a=0$, the phase remainder is of higher order:

\begin{corollary}[Phase remainder along small graphs]\label{cor:remainder}
	In the setting of Proposition~\ref{prop:returnexp}, for every $C_0>0$,
	\begin{equation}\label{eq:Epsiderivative}
		\sup\bigl\{|\partial_\psi\mathcal E(\psi,a,\eps)|:\psi\in J,\ \|a\|\le C_0|\eps|\bigr\}=o(\eps).
	\end{equation}
	If $a_\eps:J\to\R^2$ is Lipschitz with $\|a_\eps\|_{\infty}+\Lip(a_\eps)=O(\eps)$, then
	\begin{equation}\label{eq:EcomposedLip}
		\|\mathcal E(\cdot,a_\eps(\cdot),\eps)\|_{C^0(J)}=O(\eps^2),
		\qquad
		\Lip\bigl(\mathcal E(\cdot,a_\eps(\cdot),\eps)\bigr)=o(\eps),
	\end{equation}
	and if $a_\eps$ is $C^1$ with $\|a_\eps\|_{C^1(J)}=O(\eps)$, then
	\begin{equation}\label{eq:Ecomposed}
		\|\mathcal E(\cdot,a_\eps(\cdot),\eps)\|_{C^1(J)}=o(\eps).
	\end{equation}
\end{corollary}

\begin{proof}[Proof of Corollary~\ref{cor:remainder}]
	We use the decomposition $\mathcal E=\eps[G(\psi,a)-G(\psi,0)]+\widetilde{\mathcal E}$ from the preceding proof. Since $DG$ is uniformly continuous on the tube, it gives~\eqref{eq:Epsiderivative}. If $\|a_\eps\|_\infty+\Lip(a_\eps)=O(\eps)$, then $G(\psi,a_\eps(\psi))-G(\psi,0)$ has uniform norm $O(\eps)$, and its Lipschitz constant is $o(1)$ because $DG$ is uniformly continuous on the tube and $a_\eps\to0$ uniformly. Multiplying by $\eps$ and adding $\widetilde{\mathcal E}$ gives~\eqref{eq:EcomposedLip}. If $a_\eps$ is $C^1$ with $\|a_\eps\|_{C^1}=O(\eps)$, the same calculation with derivatives gives~\eqref{eq:Ecomposed}.
\end{proof}

\subsection{Proof of Theorem~\ref{thm:reciprocallocal} and Corollary~\ref{cor:period}}
\begin{proof}[Proof of Theorem~\ref{thm:reciprocallocal}]
	Fix a compact interval $J\subset(0,2\pi)$ with $\psi_*$ in its interior. By Proposition~\ref{prop:returnexp}, the fixed points of the return are the solutions of
	\begin{align}
		0&=\eps H(\psi)+\mathcal E(\psi,a,\eps),\label{eq:fixedphaseproof}\\
		0&=(D-I)a+\eps b(\psi)+Q(\psi,a,\eps).\label{eq:fixedampproof}
	\end{align}
	Consider the second equation first. The map $\Pi(a,\psi,\eps):=(I-D)a-\eps b(\psi)-Q(\psi,a,\eps)$ satisfies $D_a\Pi(0,\psi,0)=I-D$, which is invertible because $0<\rho_i<1$. The implicit function theorem, with $\psi$ as a parameter, therefore gives an open neighbourhood of $(\psi_*,0)$ and a unique $C^1$ function
	\begin{equation}\label{eq:afix}
		a=a_\eps^{\rm fix}(\psi),
		\qquad
		\|a_\eps^{\rm fix}\|_{C^1}=O(\eps),
	\end{equation}
	that solves~\eqref{eq:fixedampproof}. Substituting it into~\eqref{eq:fixedphaseproof} and dividing by $\eps\neq0$ gives, by~\eqref{eq:EcomposedLip} and~\eqref{eq:Ecomposed},
	\begin{equation}\label{eq:scalarfix}
		0=H(\psi)+r_\eps(\psi),
		\qquad
		\|r_\eps\|_{C^0}=O(\eps),
		\qquad
		\|r_\eps\|_{C^1}=o(1).
	\end{equation}
	Choose a closed interval $I$ with $\psi_*$ in its interior and a constant $c>0$ such that $|H'|\ge c$ on $I$. For small $\eps$ we have $|r_\eps'|\le c/2$ on $I$, so $H+r_\eps$ is strictly monotone there, and it has opposite signs at the endpoints of $I$ because $H$ does and $\|r_\eps\|_{C^0}=O(\eps)$. Hence $H+r_\eps$ has exactly one zero $\psi_\eps$ in $I$, and by the mean value theorem
	\begin{equation}\label{eq:phasefixedproof}
		|\psi_\eps-\psi_*|\le \frac{2}{c}\,|r_\eps(\psi_*)|=O(\eps).
	\end{equation}
	With~\eqref{eq:afix}, this is the unique fixed point of the return near $(\psi_*,0)$.
	
	At this fixed point $a=O(\eps)$, and Proposition~\ref{prop:returnexp} gives
	\begin{equation}\label{eq:blockderivative}
		D\mathcal P_\eps=
		\begin{pmatrix}
			1+\eps H'(\psi_*)+o(\eps) & O(\eps)\\
			O(\eps) & D+o(1)
		\end{pmatrix}.
	\end{equation}
	Indeed, the first entry is $1+\eps H'(\psi_\eps)+\partial_\psi\mathcal E$, where $H'(\psi_\eps)=H'(\psi_*)+o(1)$ and $\partial_\psi\mathcal E=o(\eps)$ by~\eqref{eq:Epsiderivative}; the remaining entries are $D_a\mathcal E$, $\eps b'+\partial_\psi Q$ and $D+D_aQ$, bounded by~\eqref{eq:Eaderivative} and~\eqref{eq:Qderivatives}.
	
	Write $M_{11}$, $M_{12}$, $M_{21}$, $M_{22}$ for the four blocks. At $\eps=0$ the multiplier $1$ is separated from $\rho_1$ and $\rho_2$, so for $\lambda$ near $1$ the matrix $\lambda I-M_{22}$ is uniformly invertible and
	\[
	\det(\lambda I-D\mathcal P_\eps)
	=\det(\lambda I-M_{22})
	\bigl[\lambda-M_{11}-M_{12}(\lambda I-M_{22})^{-1}M_{21}\bigr].
	\]
	Since $M_{12}$ and $M_{21}$ are $O(\eps)$, the last product is $O(\eps^2)$ uniformly near $\lambda=1$. Hence the eigenvalue near $1$ is
	\begin{equation}\label{eq:relmultproof}
		\lambda_{\rm rel}
		=M_{11}+O(\eps^2)
		=1+\eps H'(\psi_*)+o(\eps),
	\end{equation}
	and the other two converge to $\rho_1$ and $\rho_2$. For small $\eps>0$ all three lie in the open unit disc if $H'(\psi_*)<0$, and $\lambda_{\rm rel}>1$ if $H'(\psi_*)>0$. By Definition~\ref{def:locking} this proves the stability statements.
\end{proof}

\begin{proof}[Proof of Corollary~\ref{cor:period}]
	The multipliers were obtained in the proof of Theorem~\ref{thm:reciprocallocal}. Over one period $\mathcal T_\eps$ of the locked state, the phase of oscillator~1 increases by $\omega\mathcal T_\eps$ along the flow and by $\eps h_1(-\psi_*)+O(\eps^2)$ through the impulse it receives, by~\eqref{eq:C2Taylor}, \eqref{eq:phasefixedproof} and $h_1\in C^1$. The total is $2\pi$:
	\[
	2\pi=\omega\mathcal T_\eps+\eps h_1(-\psi_*)+O(\eps^2),
	\]
	which is the first expression in~\eqref{eq:period}. The same calculation for oscillator~2 gives the second. The fixed point is hyperbolic because $\lambda_{\rm rel}\neq1$ by~\eqref{eq:relmultproof} and the other two multipliers converge to $\rho_1,\rho_2\in(0,1)$. Persistence is the continuation of a hyperbolic fixed point of a $C^2$ return map that depends continuously on the data in the $C^1$ topology.
\end{proof}

\subsection{The winding argument}

\begin{lemma}[Relative winding forces phase contacts]\label{lem:winding}
	Let $\Gamma$ be a planar limit cycle with phase gradient $Z$ and tangent field $F$, let $W$ be a continuous nonvanishing vector field along $\Gamma$, and put
	\[
	\Psi(\theta):=Z(\theta)\cdot W(\gamma(\theta)),
	\qquad
	m:=\wind_\Gamma(F)-\wind_\Gamma(W).
	\]
	If $m\neq0$, then $\Psi$ takes both positive and negative values and has at least $2|m|$ zeros. If moreover $\Psi$ is $C^1$ and every zero of $\Psi$ is simple, then the zeros are finite in number and equally many are downward and upward; in particular there are at least $|m|$ of each.
\end{lemma}

\begin{proof}
	By~\eqref{eq:normalisation}, $Z(\theta)\cdot F(\gamma(\theta))>0$ for every $\theta$. The straight-line homotopy
	\[
	(1-s)\frac{F}{\|F\|}+s\frac{Z}{\|Z\|},
	\qquad 0\le s\le1,
	\]
	therefore never vanishes, and $Z$ has the same winding number as $F$.
	
	Let $\alpha:\R\to\R$ be a continuous lift of the angle from $W$ to $Z$. Then
	\[
	\alpha(\theta+2\pi)-\alpha(\theta)=2\pi m,
	\qquad
	\Psi(\theta)=\|Z(\theta)\|\,\|W(\gamma(\theta))\|\cos\alpha(\theta).
	\]
	Since $m\neq0$, $\alpha$ is not constant, so the two fields are not perpendicular at every phase and we may choose the phase origin with $\Psi(0)\neq0$. Over $[0,2\pi]$ the lift then runs from $\alpha(0)$ to $\alpha(0)+2\pi m$, neither of which is of the form $\pi/2+k\pi$, and the open interval between them contains exactly $2|m|$ such values. By the intermediate value theorem $\alpha$ attains each of them, at distinct phases in $(0,2\pi)$, and $\Psi$ vanishes there. The same interval contains values at which $\cos\alpha=1$ and values at which $\cos\alpha=-1$, so $\Psi$ takes both signs.
	
If $\Psi$ is $C^1$ and every zero is simple, then $\Psi$ changes sign at each zero, and the zeros are finite in number (Remark~\ref{rem:finitezeros}). Around the circle they are therefore alternately downward and upward, so there are equally many of each. Since there are at least $2|m|$ in all, at least $|m|$ are downward and at least $|m|$ upward.
\end{proof}

\begin{proof}[Proof of Theorem~\ref{thm:reciprocalgeometric}]
	Apply Lemma~\ref{lem:winding} to $\Gamma_2$, with $Z=Z_2$ and $W$ given by~\eqref{eq:effectivefield}, so that the integer $m$ of the lemma is~\eqref{eq:windingnumber}; by~\eqref{eq:effectiveidentity}, $\Psi=H$. Hence $H$ has at least $|m|$ downward and at least $|m|$ upward zeros. Each of them other than $\psi=0$ lies in $(0,2\pi)$, and Theorem~\ref{thm:reciprocallocal} applies to it. If $H(0)\neq0$ this covers every zero, which gives at least $|m|$ attracting and at least $|m|$ unstable locked states.
\end{proof}

\subsection{Generic constant translations}
\begin{lemma}[A constant translation produces advances and delays]\label{lem:translation}
	Let $\Gamma$ be a hyperbolic attracting planar limit cycle with phase gradient $Z$. For every $d\in\R^2\setminus\{0\}$ the function
	\[
	h_d(\theta):=Z(\theta)\cdot d
	\]
	takes both positive and negative values. For $d/\|d\|$ in an open, dense subset of the unit circle with complement of measure zero, every zero of $h_d$ is simple.
\end{lemma}

\begin{proof}
	The constant field $d$ has winding number $0$ and the tangent field has winding number $\pm1$. Lemma~\ref{lem:winding} with $W=d$ therefore applies with $m=\pm1$, and $h_d$ takes both signs.
	
	Write $d_\alpha=(\cos\alpha,\sin\alpha)$ and consider the $C^1$ function
	\[
	\Lambda(\theta,\alpha):=Z(\theta)\cdot d_\alpha
	\qquad\text{on }\Sone\times\Sone.
	\]
	At a zero of $\Lambda$ the vector $Z(\theta)$ is perpendicular to $d_\alpha$, and hence
	\[
	\partial_\alpha\Lambda(\theta,\alpha)
	=Z(\theta)\cdot(-\sin\alpha,\cos\alpha)\neq0.
	\]
	Thus $0$ is a regular value and $\Lambda^{-1}(0)$ is a compact $C^1$ curve in the torus \cite{Lee2013}. A direction $\alpha$ is a regular value of the projection of this curve to the $\alpha$-coordinate exactly when $\partial_\theta\Lambda\neq0$ at every zero, that is, when every zero of $h_{d_\alpha}$ is simple. By Sard's theorem the critical values have measure zero, and they form a closed set because the curve is compact. The regular directions are therefore open and of full measure, hence dense.
\end{proof}

\begin{proof}[Proof of Theorem~\ref{thm:opencouplings}]
	Choose $d_2$ such that all zeros of $h_2=Z_2\cdot d_2$ are simple and $h_2(0)\neq0$. By Lemma~\ref{lem:translation} this holds for an open, dense set of directions with complement of measure zero, the second condition excluding only the two directions perpendicular to $Z_2(0)$. Since $h_2$ changes sign and its zeros are simple, it has a downward zero $\psi_s\in(0,2\pi)$.
	
	With $d_1$ replaced by $\eta d_1$, $\eta>0$, the function $H$ becomes
	\[
	H_\eta(\psi)=h_2(\psi)-\eta h_1(-\psi),
	\]
	and $H_\eta\to h_2$ in $C^1(\Sone)$ as $\eta\to0^+$. By the implicit function theorem, $\psi_s$ continues to a zero $\psi_s(\eta)$ of $H_\eta$ for small $\eta$. After decreasing $\eta_0$ if necessary, these zeros lie in a compact interval $J\subset(0,2\pi)$ and satisfy
	\[
	H_\eta'(\psi_s(\eta))\le -c<0,
	\qquad 0\le\eta\le\eta_0.
	\]
	The estimates of Proposition~\ref{prop:returnexp} are uniform on this compact set of phases and ratios $\eta$. Hence a single $\eps_0>0$ can be chosen for which Theorem~\ref{thm:reciprocallocal} gives an orbitally asymptotically stable locked state for all $0<\eta<\eta_0$ and $0<\eps<\eps_0$.
	
	Simplicity of a zero and the sign of its derivative are open conditions in $C^1$, and the compactness argument is uniform under small changes of $d_1$, $d_2$ and $\eta$. This gives the persistence statement.
\end{proof}

\section{Global circle dynamics}\label{sec:global}

Theorems~\ref{thm:reciprocallocal}--\ref{thm:opencouplings} produce locked states one at a time, at phase differences where the order of the emissions is fixed. We now organise the dynamics for all phase differences on a circle. The only new difficulty lies near simultaneous emission, where the order of the emissions can change and a kick may carry the receiver across its own section.

A map $R$ between subsets of Euclidean space is \emph{Lipschitz} if $\|R(x)-R(y)\|\le C\|x-y\|$ for some $C>0$ and all $x,y$, and \emph{bi-Lipschitz onto its image} if, in addition, $\|R(x)-R(y)\|\ge c\|x-y\|$ for some $c>0$. For a bounded Lipschitz map $f$ on a compact set we write $\|f\|_{\mathrm{Lip}}:=\|f\|_\infty+\Lip(f)$. A continuous map is \emph{piecewise $C^1$} if its domain is divided into finitely many regions, on each of which the sequence of events is fixed, and its restriction to each region extends $C^1$ to the boundary of that region; the one-sided derivatives need not agree.

\subsection{A circle-map lemma}
Let $\pi:\R\to\Sone=\R/(2\pi\Z)$ be the quotient map. A \emph{lift} of an orientation-preserving circle homeomorphism $q$ is a strictly increasing map $\widetilde q:\R\to\R$ with
\[
\pi\circ\widetilde q=q\circ\pi,
\qquad
\widetilde q(\theta+2\pi)=\widetilde q(\theta)+2\pi.
\]

\begin{definition}[Topologically transverse fixed point]\label{def:toptransverse}
	A fixed point $\theta_*$ of $q$ is \emph{topologically transverse} if the displacement $\widetilde q(\theta)-\theta$ has opposite signs on the two sides of $\theta_*$. It is \emph{attracting} if the displacement is positive immediately to the left of $\theta_*$ and negative immediately to the right, and \emph{repelling} in the opposite case.
\end{definition}

\begin{definition}[Topologically Morse--Smale circle dynamics]\label{def:topMS}
	An orientation-preserving circle homeomorphism has \emph{topologically Morse--Smale dynamics} if it has a nonempty finite set of fixed points, all topologically transverse; attracting and repelling fixed points alternate around the circle; every other orbit converges forward to an attracting fixed point and backward to a repelling one; and there are no other nonwandering points.
\end{definition}

The next lemma needs a uniform first-order expansion everywhere, but a Lipschitz estimate only near the zeros of $H$, which the hypothesis $H(0)\neq0$ keeps away from $0$.

\begin{lemma}[Simple phase balances away from simultaneous emission]\label{lem:circle}
	Let $H\in C^1(\Sone)$ have at least one zero, all of them simple, and $H(0)\neq0$, and let $J_*\subset(0,2\pi)$ be a compact interval whose interior contains all zeros of $H$. Let $q_\eps$ be orientation-preserving circle homeomorphisms with lifts
	\begin{equation}\label{eq:circleexp}
		\widetilde q_\eps(\theta)
		=\theta+\eps H(\theta)+r_\eps(\theta),
		\qquad
		\|r_\eps\|_\infty=o(\eps),
		\qquad
		\Lip(r_\eps|_{J_*})=o(\eps).
	\end{equation}
	Then, for every sufficiently small $\eps>0$, $q_\eps$ has topologically Morse--Smale dynamics, with exactly one fixed point near each zero of $H$ and no others, attracting at downward zeros and repelling at upward zeros.
\end{lemma}

\begin{proof}
	Since $H(0)\neq0$, there is a closed neighbourhood $U$ of $0$ on which $|H|\ge c_0>0$. By the uniform bound in~\eqref{eq:circleexp} the displacement $\widetilde q_\eps(\theta)-\theta$ has the sign of $H$ on $U$ for small $\eps$, so $q_\eps$ has no fixed point in $U$.
	
	The zeros of $H$ are finite in number (Remark~\ref{rem:finitezeros}). Around each zero $\theta_*$ choose a closed interval $I_{\theta_*}$ in the interior of $J_*$, with $\theta_*$ in its interior, on which $|H'|\ge c_*>0$, these intervals being pairwise disjoint. By the Lipschitz bound in~\eqref{eq:circleexp} the displacement is strictly monotone on $I_{\theta_*}$ for small $\eps$, and by the uniform bound it has the signs of $H$ at the endpoints. It therefore has exactly one zero there, attracting if $H'(\theta_*)<0$ and repelling if $H'(\theta_*)>0$.
	
	On the complement of $U$ and of the interiors of the intervals $I_{\theta_*}$, which is compact, $|H|$ has a positive minimum, and the $o(\eps)$ remainder allows no further fixed point. Between consecutive fixed points the displacement has constant sign, and since $q_\eps$ preserves orientation every orbit there moves monotonically from the repelling endpoint to the attracting one. This gives the Morse--Smale dynamics.
 
\end{proof}

\begin{lemma}[Circle conjugacy from the fixed-point pattern]\label{lem:circleconjugacy}
	Two orientation-preserving circle homeomorphisms with topologically Morse--Smale dynamics are topologically conjugate if they have the same number of fixed points, with the same cyclic order of attracting and repelling types.
\end{lemma}

\begin{proof}
	Let $q$ and $\widehat q$ be the two maps, and match their fixed points in cyclic order. On each open interval between consecutive fixed points both maps have no fixed point and move every point monotonically from the repelling endpoint to the attracting one. Choose a point $x$ in such an interval and a point $\widehat x$ in the corresponding one, choose an increasing homeomorphism from the interval between $x$ and $q(x)$ onto the interval between $\widehat x$ and $\widehat q(\widehat x)$, and extend it to all iterates. Monotone convergence to the endpoints gives continuity there. The conjugacies on the intervals agree at the fixed points and glue to a circle homeomorphism.
\end{proof}

\subsection{Periodic external forcing and the master--slave case}\label{sec:strobe}
For the smooth maps of this subsection, an invariant circle is \emph{normally attracting} if its normal bundle is invariant and uniformly contracted, while the dynamics along the circle is neutral or contracts more weakly. In the planar setting this is the standard normally hyperbolic invariant circle.

Let $\dot x=F(x)$ be a planar oscillator with a hyperbolic attracting cycle $\Gamma$ of period $T$, and let $K^\eps$ be a $C^2$ impulse map with first-order field $B$, applied once every time $T$ by an external periodic source or by a master oscillator unaffected by the receiver. Sampling immediately before successive impulses gives the smooth stroboscopic map
\[
P_\eps:=\phi_T\circ K^\eps.
\]
At $\eps=0$, $P_0=\phi_T$ fixes every point of $\Gamma$, with tangent multiplier $1$ and transverse multiplier $\rho\in(0,1)$, so $\Gamma$ is normally attracting.

\begin{theorem}[Stroboscopic Morse--Smale locking]\label{thm:strobe}
	For all sufficiently small $\eps$, the map $P_\eps$ has a $C^1$ normally attracting invariant circle $\Gamma_\eps$ in a small tubular neighbourhood of $\Gamma$. Let
	\[
	h(\theta):=\dd\Theta_{\gamma(\theta)}\bigl(B(\gamma(\theta))\bigr)
	\]
	be the first-order phase response. If $h$ has at least one zero and all its zeros are simple, then for every sufficiently small $\eps\neq0$ the restriction $P_\eps|_{\Gamma_\eps}$ is a $C^1$ Morse--Smale circle diffeomorphism with exactly one fixed point near each zero of $h$. Its attracting fixed points correspond to the downward zeros of $h$ if $\eps>0$, and to the upward zeros if $\eps<0$.
\end{theorem}

\begin{proof}
	Normal hyperbolicity gives the invariant circle and its stable foliation. In the phase coordinate induced by $\Theta$, the circle map has a lift $\widetilde q_\eps(\theta)=\theta+\eps h(\theta)+o_{C^1}(\eps)$. As in the proof of Theorem~\ref{thm:reciprocallocal}, for small $\eps\neq0$ it has exactly one fixed point near each of the finitely many zeros $\theta_*$ of $h$ (Remark~\ref{rem:finitezeros}) and no others, with derivative $1+\eps h'(\theta_*)+o(\eps)$ there. Each fixed point is therefore hyperbolic, attracting if $\eps h'(\theta_*)<0$ and repelling if $\eps h'(\theta_*)>0$, and a circle diffeomorphism with a nonempty finite set of fixed points, all hyperbolic, is Morse--Smale.
\end{proof}

\subsection{Reciprocal return through nearly simultaneous emission}\label{sec:reciprocalglobal}

The local results of Section~\ref{sec:returnproof} do not require a convention at simultaneous emission. A global return does: a kick can move the receiver across its own emission section, and a literal return to the single section $\mathcal S_1$ can have different one-sided limits as the order of emission changes. We first state a condition on the \emph{paired-event transition}, establish the global dynamics under that condition, and then prove that exact invariance of the emission sections is sufficient for it. The regularity condition is not imposed on the impulse maps individually.

\paragraph{Events and the local transition.}
Near $p_i$ choose a $C^2$ defining function $g_i$ of $\Sigma_i$ with $\dd g_i(F_i)>0$. Thus incoming and outgoing points have $g_i<0$ and $g_i>0$, respectively. Fix $\tau_*>0$ smaller than a uniform positive lower bound for the times between consecutive positively oriented natural crossings near either cycle. Take $|\eps|$ so small that any short recrossing caused by an $O(|\eps|)$ kick occurs within $\tau_*$. Such recrossings take $O(|\eps|)$ time by transversality, whereas successive genuine emissions are separated by a time bounded below independently of $\eps$.

\begin{definition}[Single-emission convention]\label{def:singleemission}
Oscillator $i$ is \emph{eligible} to emit if at least $\tau_*$ has elapsed since its last recorded emission. A positively oriented crossing of $\Sigma_i$ is counted as an emission if oscillator $i$ is eligible. If an eligible receiver is carried by a kick from the incoming side onto or across $\Sigma_i$, its emission is recorded at the instant of the kick. A recrossing within $\tau_*$ of a recorded emission is ignored. At exact simultaneity both emissions are recorded, and the two impulses act on the pre-impulse states as in~\eqref{eq:simultaneous}.
\end{definition}

On the return domains below, the last emission of each oscillator is determined by the return episode. Consequently the convention does not introduce an independent memory coordinate. A \emph{paired episode} is a segment of trajectory containing exactly
two recorded emissions, one from each oscillator, possibly at the same instant. The segment begins before both emissions and ends after both,
before the next recorded emission of either oscillator.

Choose an incoming three-dimensional section $\mathcal S_B$ preceding the two potentially simultaneous emissions and an outgoing section $\mathcal S_{\rm out}$ following them but preceding the next emissions. The map $E_\eps:\mathcal S_B\to\mathcal S_{\rm out}$ obtained by traversing one paired episode and then following the free flow to $\mathcal S_{\rm out}$ is the \emph{paired-event transition}. Its restrictions to the two event orders will be denoted $E_\eps^+$ (oscillator~1 emits first) and $E_\eps^-$ (oscillator~2 emits first). The local switching set is the boundary between these two orders. The sections and their compactly contained coordinate domains are fixed independently of $\eps$.

\begin{definition}[Regular paired-event transition]\label{def:pairedregular}
The paired-event transition is \emph{regular} if, on a fixed neighbourhood of the simultaneous-emission state and for all sufficiently small $|\eps|$, it is piecewise $C^1$ and, in fixed incoming and outgoing coordinates,
\begin{equation}\label{eq:pairedReg}
 \|E_\eps-E_0\|_\infty+\Lip(E_\eps-E_0)\le C|\eps|.
\end{equation}
Since $E_0$ is smooth, \eqref{eq:pairedReg} implies that $E_\eps$ is Lipschitz, and in particular continuous across the switching set. Since $E_0$ is a local diffeomorphism between transverse sections, after shrinking the neighbourhood it has a lower Lipschitz bound $c_0>0$, and \eqref{eq:pairedReg} gives $\|E_\eps(z)-E_\eps(z')\|\ge(c_0-C|\eps|)\|z-z'\|$; thus $E_\eps$ is bi-Lipschitz onto its image for small $|\eps|$. The one-sided derivatives are not required to coincide.
\end{definition}

The regularity condition does not require preservation of the emission sections: Example~C in Section~\ref{sec:analytic} provides an explicit regular transition for which neither emission section is invariant.
Conversely, \eqref{eq:pairedReg} is a restrictive compatibility requirement, not a consequence of smooth impulses. Example~A in Section~\ref{sec:analytic} shows that it fails throughout an open three-parameter family of smooth impulse laws.

On the incoming section $\mathcal S_B$ constructed in the proof of the next lemma, let $t_i(z)$ be the two \emph{free} first-hit times of $\Sigma_i$ from $z\in\mathcal S_B$ and $u(z):=t_2(z)-t_1(z)$. The \emph{switching set} is $\{u=0\}$ and, for a fixed small $u_0>0$, the \emph{switching flow box} is $\{z\in\mathcal S_B:|u(z)|<u_0\}$. Let $s_{12}\ge0$ (for $u\ge0$) and $s_{21}\ge0$ (for $u\le0$) be the time between the two recorded emissions, $w_i^\eps$ the state of receiver $i$ immediately before its kick, and $w_i^0$ the corresponding receiver state for the free first-hit times.

\begin{lemma}[Uniform event-time and phase bookkeeping]\label{lem:pairedepisode}
Under Hypotheses~\ref{hyp:cycles} and~\ref{hyp:equalperiods} and the impulse law of Section~\ref{sec:setting}, with emissions counted as in Definition~\ref{def:singleemission}, the switching set is a smooth hypersurface independent of $\eps$ and each episode records exactly one emission of each oscillator. Uniformly on $\mathcal S_B$, including as $u\to0$,
\begin{equation}\label{eq:uniformTimes}
 |s_{12}-u|\le C|\eps|\ \ (u\ge0),
 \qquad
 |s_{21}+u|\le C|\eps|\ \ (u\le0),
 \qquad
 s_{12}=s_{21}=0\ \ (u=0),
\end{equation}
\begin{equation}\label{eq:uniformReceivers}
 \|w_i^\eps-w_i^0\|\le C|\eps|.
\end{equation}
Outside the switching flow box the order of the emissions is fixed, and no kick carries its receiver onto, across or back across its section.
\end{lemma}

\begin{proof}
Take real lifts $\theta_i$ of the asymptotic phases on a neighbourhood of the arcs $|\theta_i|\le3\pi/4$ of the cycles, continued along trajectories, put $\psi=\theta_2-\theta_1$ and $\chi=(\theta_1+\theta_2)/2$, and choose $0<2\kappa<\chi_0<\pi/2$. An incoming section is the level $\chi=-\chi_0$ with $|\psi|<3\kappa$, restricted to a small transverse disc. It is transverse to the product free flow because $\dot\chi=\omega$. The larger interval leaves room for the buffered extension of the incoming chart constructed in the proof of Lemma~\ref{lem:gluing}. The free hits of $\Sigma_i$ have $C^2$ times $t_i(z)$ by the implicit function theorem and the positive lower bound for $\dd g_i(F_i)$ on compact subsegments. On the product of the two cycles,
\[
 t_1=(\chi_0+\psi/2)/\omega,\qquad
 t_2=(\chi_0-\psi/2)/\omega,\qquad
 u=-\psi/\omega.
\]
Thus $\dd u\ne0$ near the simultaneous state after reducing the transverse disc. No kick occurs before the first of the two emissions, so $u=0$ is independent of $\eps$. Since $3\kappa/2<\chi_0$, the two emissions occur a fixed positive time after the level $\chi=-\chi_0$ and before the next level $\chi=2\pi-\chi_0$.

At $u=0$ the two free hits coincide; by Definition~\ref{def:singleemission} both emissions are recorded at that instant, the impulses act as in~\eqref{eq:simultaneous}, and $s_{12}=s_{21}=0$, so the estimates below hold trivially there. At $u=0$ the pre-impulse receiver states coincide with their free counterparts, $w_i^\eps=w_i^0$. Suppose $u>0$, so oscillator 1 emits first. Its kick displaces receiver 2 by $O(|\eps|)$ in $C^0$. In a local signed-time extension, the transverse hitting equation
\[
 g_2\bigl(\phi^2_{\widehat s_{12}}K_2^\eps\phi^2_{t_1(z)}(x_2)\bigr)=0
\]
has a unique $C^2$ solution satisfying $\widehat s_{12}(z,0)=u(z)$ and $|\widehat s_{12}(z,\eps)-u(z)|\le C|\eps|$; the derivative with respect to the hitting time is bounded away from zero. If $\widehat s_{12}>0$, the receiver remains incoming after the kick and emits at that hit. If $\widehat s_{12}\le0$, it is carried onto or through its section by the kick and, by Definition~\ref{def:singleemission}, emits immediately. Hence
\[
 s_{12}=\max\{0,\widehat s_{12}\},\qquad
 |s_{12}-u|\le|\widehat s_{12}-u|\le C|\eps|,
\]
since $u\ge0$. The other order is identical, with $-u$ in place of $u$. Subsequent genuine emissions are a fixed positive time away, while rapid recrossings, if any, are suppressed by the convention. This establishes the event count. The actual reception times differ from their free counterparts by at most $C|\eps|$; smoothness and boundedness of the free flows, together with the $O(|\eps|)$ size of the kicks, give~\eqref{eq:uniformReceivers}.

Outside the switching flow box, $|u|\ge u_0$. Suppose $u\ge u_0$. The first kick then reaches receiver 2 at least a time $u_0$ before its free hit of $\Sigma_2$, and the second kick reaches oscillator 1 a time $s_{12}\ge u_0-C|\eps|$ after its emission. Both kicks are of size $O(|\eps|)$, so for small $|\eps|$ neither carries its receiver onto, across or back across its section. Hence $s_{12}=\widehat s_{12}$ is the $C^2$ hitting time given by the implicit function theorem as in Section~\ref{sec:sectionsimpulses}, and the estimates above hold there as well. The case $u\le-u_0$ is identical.
\end{proof}

\paragraph{The return cylinder.}
A first return to $\mathcal S_1=\Sigma_1\times V_2$ is smooth where the order of the emissions is fixed, but need not be continuous across its cut at simultaneous emission. We instead use two \emph{observation sections} and identify observations of the same hybrid trajectory. Their overlap will be chosen strictly away from switching.

Two observation sections suffice (Figure~\ref{fig:returncylinder}).

\begin{figure}[htbp]
	\centering
	\begin{tikzpicture}[>=Latex, every node/.style={font=\small}]
		\def\k{24}
		\draw[gray] (0,0) circle (1.5);
		\draw[line width=1.6pt] (\k:1.9) arc (\k:360-\k:1.9);
		\draw[line width=1.6pt] (-2*\k:2.3) arc (-2*\k:2*\k:2.3);
		\draw[line width=4pt, gray!70] (\k/2:2.3) arc (\k/2:0.75*\k:2.3);
		\fill (0:1.5) circle (1.6pt);
		\node[left] at (0:1.45) {$\psi=0$};
		\foreach \a in {\k,2*\k,-\k,-2*\k} \draw[gray, dashed] (\a:1.5) -- (\a:2.5);
		\node at (180:2.45) {A};
		\node at (0:2.65) {B};
		\node[right] at (\k:2.55) {$\kappa$};
		\node[right] at (2*\k:2.55) {$2\kappa$};
		\node[right] at (-\k:2.55) {$-\kappa$};
		\node[right] at (-2*\k:2.55) {$-2\kappa$};
		\draw[<-] (0.62*\k:2.42) -- (0.62*\k:3.6) node[right] {collar $\mathcal L$};
		\node[gray] at (135:1.15) {$\mathcal C_0$};
		\fill (100:1.5) circle (1.3pt) (215:1.5) circle (1.3pt);
		\node[below right] at (215:1.45) {\footnotesize zeros of $H$};
	\end{tikzpicture}
	\caption{The two charts of the return cylinder over the circle $\mathcal C_0$ of phase differences: chart A, a fixed-order section over $\kappa<\psi<2\pi-\kappa$, and chart B, which contains simultaneous emission $\psi=0$. They overlap in two buffered strips near $\psi=\pm\kappa$, strictly away from switching, whose enlarged domains absorb the $O(\eps)$ phase change of one return; the coordinates are interpolated in the collar $\mathcal L$ of chart B. When $H(0)\neq0$, $\kappa$ is chosen so that all zeros of $H$ lie in chart A.}
	\label{fig:returncylinder}
\end{figure}
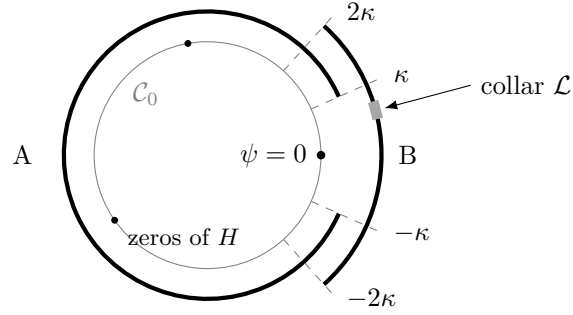

\begin{lemma}[Compatible observations]\label{lem:gluing}
Under the single-emission convention and the regularity of Definition~\ref{def:pairedregular}, for sufficiently small $|\eps|$ there are a smooth fixed-order return $\mathcal P_\eps$ on a section $\mathcal S_A$, a paired-episode return $\mathcal R_{B,\eps}$ on a section $\mathcal S_B$, and a $C^1$ gluing map $\sigma_\eps$, defined on two fixed-order overlap strips and a local diffeomorphism there, such that
\begin{equation}\label{eq:returnconjugacy}
 \mathcal P_\eps\circ\sigma_\eps
 =\sigma_\eps\circ\mathcal R_{B,\eps}
\end{equation}
wherever both sides are defined.
\end{lemma}

\begin{proof}
We specify the domains to control iterations that change charts. Choose $\kappa\in(0,\pi/8)$, then $\chi_0\in(2\kappa,\pi/2)$ and $0<\zeta<\kappa/20$. The fixed-order chart A is the part of $\mathcal S_1$ with $\kappa<\psi<2\pi-\kappa$; chart B is the incoming section $\chi=-\chi_0$ with $|\psi|<2\kappa$. In either chart use a sufficiently small transverse disc. The chart A coordinates are $X_A=(\psi,a_1,a_2)$, as in~\eqref{eq:tubecoordinates}. Its coordinate map extends through $\psi=0$, although its physical fixed-order return need not.

Use enlarged phase intervals
\[
 A^+:\ \kappa-3\zeta<\psi<2\pi-\kappa+3\zeta,
 \qquad
 B^+:\ |\psi|<2\kappa+3\zeta
\]
and enlarged transverse discs. The inner charts have transverse radius $\ell_0$ and the extensions radius $\ell_1>\ell_0$. All event times and return formulas are defined on neighbourhoods of the closed inner domains, by Lemma~\ref{lem:pairedepisode} and the fixed-order hitting-time theorem. The image of either inner chart under one return stays in its enlarged chart for $|\eps|$ small: the angular displacement is $O(|\eps|)$, since there are two kicks of size $O(|\eps|)$ and free flight preserves $\psi$, while the transverse images remain in the larger disc after reducing $\ell_0$. Thus an iteration that leaves an inner chart can still be evaluated in its extension. Chart B and its extension $B^+$ lie in the incoming section of Lemma~\ref{lem:pairedepisode}, since $2\kappa+3\zeta<3\kappa$. Fix the constant $u_0$ of that lemma so small that $u_0<\kappa/(8\omega)$ and that, after reducing the transverse discs, the set $\{|u|<2u_0\}$ lies in the neighbourhood of Definition~\ref{def:pairedregular}. Since $u=-\psi/\omega$ on the product of the cycles, the transverse discs can be reduced so that $|u+\psi/\omega|<\kappa/(8\omega)$ on the enlarged B-domain. The strips $\kappa/4<|\psi|<2\kappa+3\zeta$ therefore satisfy $|u|>\kappa/(8\omega)>u_0$ and lie outside the switching flow box, where the order of the emissions is fixed.

On $\mathcal S_B$, the local return $\mathcal R_{B,\eps}$ follows the trajectory from the lifted level $\chi=-\chi_0$ to its next level $\chi=2\pi-\chi_0$. The two emissions lie a fixed positive time away from the entry and exit levels. Since the kicks change $\chi$ by only $O(|\eps|)$, this return traverses exactly one paired episode. Near switching, where $|u|<2u_0$, it is the regular transition followed by a smooth free flight; where $|u|\ge u_0$ it is a smooth fixed-order map, by Lemma~\ref{lem:pairedepisode}. It is therefore continuous, piecewise $C^1$, bi-Lipschitz, and $O(|\eps|)$-close in Lipschitz norm to its unperturbed counterpart. On chart A the return $\mathcal P_\eps$ is the smooth fixed-order map of Section~\ref{sec:localreturn}.

For $\kappa/4<|\psi|<2\kappa+3\zeta$, define $\sigma_\eps$ by transporting an observation on $\mathcal S_B$ to the \emph{next} emission of oscillator 1, immediately before the event. If $\psi<0$, oscillator 1 emits first and $\sigma_\eps=\sigma_0$ is a free-flow hitting map. If $\psi>0$, oscillator 2 emits first and $\sigma_\eps$ is the free flight to $\Sigma_2$, followed by $K_1^\eps$ and the free flight to $\Sigma_1$. On both overlap strips the event order is fixed uniformly in $\eps$ and in the transverse coordinates, once the latter have been restricted. Consequently $\sigma_\eps$ is a local diffeomorphism and
\begin{equation}\label{eq:sigmaclose}
 \|\sigma_\eps-\sigma_0\|_{C^1}\le C|\eps|.
\end{equation}
The maps in~\eqref{eq:returnconjugacy} observe the same trajectory at the same next emission of oscillator 1: the left side starts at the first such emission and takes one fixed-order return, whereas the right side takes one paired return before observing that next emission. Because the phase shift in one return is $O(|\eps|)$, their common compact overlap can be reduced slightly so that both compositions are defined there for small $|\eps|$. The equality is exact, including the transverse components.
\end{proof}

\begin{lemma}[A global return]\label{lem:returncylinder}
Under the assumptions of Lemma~\ref{lem:gluing}, $\sigma_\eps$ is a diffeomorphism onto its image where the charts are glued, and the quotient $\mathcal M_\eps$ of the two sections by $\sigma_\eps$ has coordinates $\mathcal X_\eps:\mathcal M_\eps\to\Sone\times B_\ell(0)$, for some $\ell>0$, with uniformly bi-Lipschitz $C^1$ coordinate changes and with angular coordinate equal to the physical relative phase on chart A. In these coordinates the two returns descend to a continuous, bi-Lipschitz, piecewise $C^1$ return $\mathcal R_\eps=(Q_\eps,A_\eps)$ on a smaller product tube, with
\begin{align}
 &\|\mathcal R_\eps-\mathcal R_0\|_\infty
 +\Lip(\mathcal R_\eps-\mathcal R_0)=O(|\eps|),\label{eq:globalLipclose}\\
 &Q_0(\vartheta,a)=\vartheta,\qquad A_0(\vartheta,0)=0,\qquad
 D_a A_0(\vartheta,0)=\diag(\rho_1,\rho_2).\label{eq:uncoupledGlobal}
\end{align}
\end{lemma}

\begin{proof}
We use the charts, their buffered extensions $A^+$ and $B^+$, and the gluing map $\sigma_\eps$ constructed in the proof of Lemma~\ref{lem:gluing}. The coordinate map $X_A$ on $\mathcal S_1$ has a smooth continuation to phases near zero. Define $X_B^0=X_A\circ\sigma_0$ on $B^+$, where $\sigma_0$ is the free-flow map to $\mathcal S_1$, defined also at $\psi=0$. Because free flow preserves the relative phase, the angular component of $X_B^0$ is precisely $\psi$. Choose the transverse domains so that the coordinate images of both enlarged charts are products of their indicated phase intervals and a transverse disc. On the positive fixed-order strip $\psi>\kappa/4$, put
\[
 D_\eps=X_A\circ\sigma_\eps\circ(X_B^0)^{-1}-\id,
 \qquad \|D_\eps\|_{C^1}\le C|\eps|.
\]
Choose a smooth cutoff $\eta$ with $\eta=0$ for $\psi\le\kappa/2$ and $\eta=1$ for $\psi\ge 3\kappa/4$, and extend $d_\eps:=\eta D_\eps$ by zero where $\eta=0$. Then $d_\eps$ is defined on the entire enlarged B-coordinate domain and
\begin{equation}\label{eq:chartcorrection}
 \|d_\eps\|_{C^1}\le C|\eps|.
\end{equation}
Set $X_B^\eps=(\id+d_\eps)\circ X_B^0$ and retain $X_A$ on chart A. The quotient $\mathcal M_\eps$ identifies $z\in B^+$ with $\sigma_\eps(z)$ whenever $\sigma_\eps(z)\in A^+$. Since $\sigma_\eps$ changes the phase difference by $O(|\eps|)$, and not at all for $\psi<0$, an identified point $z$ satisfies either $\psi(z)>\kappa-3\zeta-O(|\eps|)$ or $\psi(z)<-\kappa+3\zeta$. In the first case $\eta(\psi(z))=1$ for small $|\eps|$, because $\zeta<\kappa/20$ gives $\kappa-3\zeta>3\kappa/4$; this is the reason for the choice of $\zeta$. Hence $X_B^\eps(z)=X_A(\sigma_\eps(z))$. In the second case $\eta=0$ and $\sigma_\eps=\sigma_0$, so the same equality holds. Thus the two coordinate maps agree exactly at every identified pair of points.

On the convex B-coordinate domain the map $\id+d_\eps$ is injective and bi-Lipschitz, since $\Lip(d_\eps)<1$. Both chart maps are individually injective; $X_A$ is injective on the whole tube in $\mathcal S_1$, including its continuation to phases near zero. Suppose that $X_B^\eps(z)=X_A(y)$ with $z\in B^+$ and $y\in A^+$. Then $\psi(y)\in(\kappa-3\zeta,2\pi-\kappa+3\zeta)$, while the angular coordinate of $X_B^\eps(z)$ differs from $\psi(z)$ by $O(|\eps|)$, and by nothing where $\eta=0$. Hence either $\psi(z)>\kappa-3\zeta-O(|\eps|)>3\kappa/4$, or $\psi(z)<-\kappa+3\zeta$. In both cases $X_B^\eps(z)=X_A(\sigma_\eps(z))$, as shown above, so $X_A(\sigma_\eps(z))=X_A(y)$, and injectivity of $X_A$ gives $\sigma_\eps(z)=y$: the two points are identified. The same identity shows that $\sigma_\eps$ is injective where the charts are glued. For $\psi<0$ it is the free-flow map $\sigma_0$, which is injective. For $\psi\ge3\kappa/4$ we have $X_B^\eps=X_A\circ\sigma_\eps$, and $X_B^\eps=(\id+d_\eps)\circ X_B^0$ is injective because both factors are; so $\sigma_\eps(z)=\sigma_\eps(z')$ implies $X_B^\eps(z)=X_B^\eps(z')$, hence $z=z'$. Being also a local diffeomorphism, $\sigma_\eps$ is a diffeomorphism onto its image there. Hence the quotient map is injective. Its restrictions are local homeomorphisms, so it is open, and its image contains a fixed smaller product neighbourhood of the central circle by compactness and $C^1$ proximity to the $\eps=0$ coordinates. This gives the coordinates $\mathcal X_\eps$, with the extended charts retained for return evaluations.

Equation~\eqref{eq:returnconjugacy} now makes the local returns a single return $\mathcal R_\eps$ on the quotient. In chart A the difference from the uncoupled return is $O(|\eps|)$ in $C^1$. In chart B it is $O(|\eps|)$ in Lipschitz norm by Definition~\ref{def:pairedregular}; the coordinate changes differ from their uncoupled counterparts by $O(|\eps|)$ in $C^1$. These estimates hold in finitely many patches covering the compact central circle. Nearby points lie in a common patch, by a Lebesgue-number argument; separated points are controlled by the uniform $O(|\eps|)$ supremum estimate. This proves~\eqref{eq:globalLipclose} on the smaller product tube.

At $\eps=0$ the phase difference is preserved on every free trajectory, so $Q_0(\vartheta,a)=\vartheta$. The central circle is pointwise fixed and the two transverse derivatives over one natural period have multipliers $\rho_i$; in the coordinates obtained from~\eqref{eq:tubecoordinates} they form $\diag(\rho_1,\rho_2)$, as stated. The uncoupled return is a diffeomorphism onto its image, and its restriction to a sufficiently small closed product tube has a uniform positive lower Lipschitz bound $c_0$. Thus~\eqref{eq:globalLipclose} implies
\[
 \|\mathcal R_\eps(z)-\mathcal R_\eps(z')\|
 \ge (c_0-C|\eps|)\|z-z'\|,
\]
which proves the claimed bi-Lipschitz property.
\end{proof}

\begin{lemma}[Global phase expansion without emission-section invariance]\label{lem:globalphaseexpansion}
Under the assumptions of Lemma~\ref{lem:returncylinder}, let $a_\eps:\Sone\to\R^2$ be Lipschitz with $\|a_\eps\|_\infty+\Lip(a_\eps)\le C|\eps|$. For a consistent lift of the angular component of the return,
\begin{equation}\label{eq:globalphaseexp}
 Q_\eps(\vartheta,a_\eps(\vartheta))
 =\vartheta+\eps H(\vartheta)+r_\eps(\vartheta),
 \qquad \|r_\eps\|_\infty=O(\eps^2),
\end{equation}
and $\Lip(r_\eps|_J)=o(|\eps|)$ on every compact interval $J\subset(0,2\pi)$ contained in the fixed-order chart A.
\end{lemma}

\begin{proof}
For local real lifts define
\[
 \Delta_i^\eps(x):=\Theta_i(K_i^\eps x)-\Theta_i(x),\qquad
 f_i(x):=\dd\Theta_i|_x(B_i(x)).
\]
The $C^2$ impulse assumption gives $\Delta_i^\eps(x)=\eps f_i(x)+O(\eps^2)$ uniformly on compact subsets. By~\eqref{eq:phaseidentity} the free flow preserves the physical phase difference $\psi=\Theta_2-\Theta_1$. Each recorded paired return contains one kick to each receiver, so its exact relative-phase increment is
\begin{equation}\label{eq:pairedPhaseIdentity}
 \psi'-\psi=\Delta_2^\eps(w_2^\eps)-\Delta_1^\eps(w_1^\eps),
\end{equation}
where $w_i^\eps$ is the receiver state immediately before its kick. By Lemma~\ref{lem:pairedepisode}, $w_i^\eps=w_i^0+O(|\eps|)$ uniformly on chart B, including where a forward kick causes an immediate emission. On the compact fixed-order chart A, the same estimate follows from the smooth hitting-time theorem. On the product of the cycles the free receiver states are $\gamma_2(\psi)$ and $\gamma_1(-\psi)$; when the transverse amplitude is $a$, their deviations from these states are $O(\|a\|)$ by smoothness. Consequently, uniformly in both charts,
\begin{equation}\label{eq:physicalphaseexp}
 \psi'-\psi
 =\eps\bigl[h_2(\psi)-h_1(-\psi)\bigr]
 +O(\eps^2+|\eps|\,\|a\|)
 =\eps H(\psi)+O(\eps^2+|\eps|\,\|a\|).
\end{equation}
Here the two physical phases refer to the beginning and end of the \emph{same representative return}. We now verify the change to global coordinates, including a return that changes chart.

For a representative $z$ in the inner B-domain, let $w=X_B^0(z)$ and $w'=X_B^0(\mathcal R_{B,\eps}z)$. These expressions are defined even if the image leaves the inner chart, because it belongs to $B^+$. Since the unperturbed B-return fixes the central circle, preserves the base phase and is Lipschitz on the tube,
\[
 \|X_B^0(\mathcal R_{B,0}z)-X_B^0(z)\|\le C\|a\|.
\]
The regularity assumption and the smooth final free flight imply
\[
 \|X_B^0(\mathcal R_{B,\eps}z)-X_B^0(\mathcal R_{B,0}z)\|
 \le C|\eps|.
\]
Thus, on a graph with $\|a\|\le C|\eps|$,
\begin{equation}\label{eq:chartmove}
 \|w'-w\|\le C(\|a\|+|\eps|)=O(|\eps|).
\end{equation}
Using~\eqref{eq:chartcorrection},
\begin{equation}\label{eq:chartOeps2}
 \|d_\eps(w')-d_\eps(w)\|
 \le\Lip(d_\eps)\|w'-w\|=O(\eps^2).
\end{equation}
The first component of~\eqref{eq:chartOeps2} is exactly the difference between the global angular increment and the physical phase increment in the B-representative. If the image passes from B to A, it is still in $B^+$ and its global coordinates computed with $X_B^\eps$ agree \emph{exactly} with $X_A\circ\sigma_\eps$ on the overlap. We have therefore not evaluated $d_\eps$ outside its domain.

For a representative in the inner A-domain, evaluate its fixed-order return in $A^+$. The angular component of $X_A$ is the physical phase $\psi$. If the image is to be represented in chart B, it lies in the overlap, where $X_A$ and $X_B^\eps$ agree exactly. In either case, on graphs of amplitude $O(|\eps|)$,
\[
 \vartheta'-\vartheta=(\psi'-\psi)+O(\eps^2).
\]
Also $\vartheta=\psi+O(|\eps|)$ in B by~\eqref{eq:chartcorrection}, and $\vartheta=\psi$ in A. Since $H\in C^1(\Sone)$, $\eps H(\psi)=\eps H(\vartheta)+O(\eps^2)$. Combined with~\eqref{eq:physicalphaseexp}, this proves~\eqref{eq:globalphaseexp}.

For the stronger estimate, take $J$ compactly contained in chart A. The local return is jointly $C^2$ there, and its phase displacement vanishes for $\eps=0$ for \emph{every} transverse coordinate. It therefore has the expansion
\[
 Q_\eps(\psi,a)-\psi=\eps G(\psi,a)+\widetilde{\mathcal E}(\psi,a,\eps),
 \quad G(\psi,0)=H(\psi),\quad
 \|\widetilde{\mathcal E}(\cdot,\cdot,\eps)\|_{C^1}=o(|\eps|).
\]
For a graph as in the statement, uniform continuity of $DG$ and $\Lip(a_\eps)=O(|\eps|)$ give
\[
 \Lip\bigl[G(\cdot,a_\eps(\cdot))-G(\cdot,0)\bigr]=o(1).
\]
Multiplying by $|\eps|$ and adding the remainder yields $\Lip(r_\eps|_J)=o(|\eps|)$.
\end{proof}

The lemma uses the regularity of the paired transition, but not invariance of $\Sigma_i$ under $K_i^\eps$.

\paragraph{Invariant circle and global locking.}
A Lipschitz invariant circle $\mathcal C$ is \emph{normally attracting in the metric sense} if, in a neighbourhood of $\mathcal C$, there are constants $C_*>0$ and $0<\lambda<1$ for which
\[
 \operatorname{dist}(\mathcal R^n z,\mathcal C)
 \le C_*\lambda^n\operatorname{dist}(z,\mathcal C)
\]
for every $z$ whose iterates remain in that neighbourhood.

\begin{theorem}[Global reciprocal invariant circle]\label{thm:globalcircle}
Under the assumptions of Lemma~\ref{lem:returncylinder}, there are constants $C_1,C_2>0$, independent of $\eps$, such that for all sufficiently small $|\eps|$ the global return $\mathcal R_\eps$ has a unique invariant graph $\mathcal C_\eps=\{(\vartheta,a_\eps(\vartheta)):\vartheta\in\Sone\}$ with $\|a_\eps\|_\infty\le C_1|\eps|$ and $\Lip(a_\eps)\le C_2|\eps|$. It is normally attracting in the metric sense, and the induced return is an orientation-preserving bi-Lipschitz circle homeomorphism whose lift $\widetilde q_\eps(\vartheta)=Q_\eps(\vartheta,a_\eps(\vartheta))$ satisfies the conclusions of Lemma~\ref{lem:globalphaseexpansion}.
\end{theorem}

\begin{proof}
Write $\mathcal R_\eps=(Q_\eps,A_\eps)$ and let $\Lip_a$ and $\Lip_\vartheta$ denote separate Lipschitz constants. Choose $\max\{\rho_1,\rho_2\}<\lambda_0<1$. By~\eqref{eq:uncoupledGlobal}, reducing the fixed tube if necessary, and then using~\eqref{eq:globalLipclose}, we have
\begin{equation}\label{eq:graphestimates}
 \Lip_a(A_\eps)\le\lambda_0,\quad
 \Lip(Q_\eps-\vartheta)\le M|\eps|,\quad
 \|A_\eps(\vartheta,0)\|\le M|\eps|.
\end{equation}
The constants are now fixed in the following order. First fix $C_1>M/(1-\lambda_0)$. Since $A_0(\vartheta,0)=0$, we have $\partial_\vartheta A_0(\vartheta,0)=0$, and $C^2$ smoothness of $A_0$ gives $\Lip_\vartheta(A_0|_{\|a\|\le C_1|\eps|})\le M'C_1|\eps|$. Adding the Lipschitz perturbation from~\eqref{eq:globalLipclose} yields
\begin{equation}\label{eq:graphangularnormal}
 \Lip_\vartheta(A_\eps|_{\|a\|\le C_1|\eps|})\le M_1|\eps|,
\end{equation}
where $M_1=M_1(C_1)$ depends on $C_1$ but not on $\eps$. Next fix $C_2>M_1/(1-\lambda_0)$. Only then is $|\eps|$ taken small.

Let $\mathcal G_\eps$ be the closed complete class of periodic Lipschitz maps $u:\Sone\to\R^2$ satisfying $\|u\|_\infty\le C_1|\eps|$ and $\Lip(u)\le C_2|\eps|$. For $u\in\mathcal G_\eps$, set $q_u(\vartheta)=Q_\eps(\vartheta,u(\vartheta))$. By~\eqref{eq:graphestimates},
\[
 \Lip(q_u-\id)\le M|\eps|(1+C_2|\eps|)=:m_\eps<1.
\]
Hence the lift of $q_u$ is increasing and $q_u$ is an orientation-preserving bi-Lipschitz circle homeomorphism with $\Lip(q_u^{-1})\le(1-m_\eps)^{-1}$. Its graph image is the graph of
\[
 (\mathsf G u)(\vartheta')
 :=A_\eps\bigl(q_u^{-1}(\vartheta'),u(q_u^{-1}(\vartheta'))\bigr).
\]
Equations~\eqref{eq:graphestimates}--\eqref{eq:graphangularnormal} imply
\[
 \|\mathsf G u\|_\infty\le M|\eps|+\lambda_0\|u\|_\infty,
 \qquad
 \Lip(\mathsf G u)\le
 \frac{M_1|\eps|+\lambda_0\Lip(u)}{1-m_\eps}.
\]
Hence $\|\mathsf G u\|_\infty\le(M+\lambda_0C_1)|\eps|\le C_1|\eps|$, by the choice of $C_1$, and $\Lip(\mathsf G u)\le(M_1+\lambda_0C_2)|\eps|/(1-m_\eps)\le C_2|\eps|$ as soon as $m_\eps\le1-\lambda_0-M_1/C_2$. The right-hand side is positive by the choice of $C_2$, and $m_\eps\to0$ for the fixed $C_2$, so this holds for all sufficiently small $|\eps|$. Thus $\mathsf G$ preserves $\mathcal G_\eps$.

For $u,v\in\mathcal G_\eps$,
\[
 \|q_u^{-1}-q_v^{-1}\|_\infty
 \le\frac{M|\eps|}{1-m_\eps}\|u-v\|_\infty.
\]
Consequently, using the transverse contraction and the Lipschitz constants of $u,v$,
\begin{equation}\label{eq:graphcontraction}
 \|\mathsf G u-\mathsf G v\|_\infty
 \le(\lambda_0+C|\eps|)\|u-v\|_\infty.
\end{equation}
For $|\eps|$ small this coefficient is less than one. Banach's fixed-point theorem provides a unique invariant graph $a_\eps\in\mathcal G_\eps$.

For a point $a=a_\eps(\vartheta)+b$ near that graph, its transverse displacement after one return satisfies
\[
 \|b'\|\le
 \bigl(\Lip_a(A_\eps)
 +\Lip(a_\eps)\Lip_a(Q_\eps)\bigr)\|b\|
 \le(\lambda_0+C\eps^2)\|b\|.
\]
This proves metric normal attraction on a sufficiently small invariant tube. The induced circle map is $q_{a_\eps}$, hence is an orientation-preserving bi-Lipschitz homeomorphism. Lemma~\ref{lem:globalphaseexpansion}, applied to the fixed graph, gives the claimed phase expansion and local remainder estimate.
\end{proof}

No condition on the zeros of $H$ is required for the existence of the invariant graph.

\begin{proof}[Proof of Theorem~\ref{thm:globalreciprocal}]
The parameters are chosen in the following order. Since all zeros of $H$ are simple, they are finite in number (Remark~\ref{rem:finitezeros}), and since $H(0)\ne0$ none of them is $0$. First fix $\kappa\in(0,\pi/8)$ so small that all zeros of $H$ lie in $(2\kappa,2\pi-2\kappa)$, and put $J_*=[2\kappa,2\pi-2\kappa]$. Then construct the return cylinder of Lemmas~\ref{lem:gluing} and~\ref{lem:returncylinder} with this $\kappa$; the buffer $\zeta$, the transverse radii and $\eps_0$ are chosen afterwards, as in their proofs. The interval $J_*$ is compact, contains all zeros of $H$ in its interior, and lies in chart A, where $\kappa<\psi<2\pi-\kappa$. By Theorem~\ref{thm:globalcircle} the return possesses a normally attracting Lipschitz invariant circle, and Lemma~\ref{lem:globalphaseexpansion} gives
\[
 \widetilde q_\eps(\vartheta)
 =\vartheta+\eps H(\vartheta)+r_\eps(\vartheta),\quad
 \|r_\eps\|_\infty=O(\eps^2),\quad
 \Lip(r_\eps|_{J_*})=o(\eps).
\]
Lemma~\ref{lem:circle} therefore gives exactly one attracting fixed point near each downward zero, one repelling fixed point near each upward zero, and no other periodic or nonwandering points on the circle. Let $\psi_j$ be a zero of $H$ and $(\vartheta_\eps,a_\eps(\vartheta_\eps))$ the fixed point near it. Since $\eps H(\vartheta_\eps)+r_\eps(\vartheta_\eps)=0$ with $\|r_\eps\|_\infty=O(\eps^2)$, and $\psi_j$ is simple, $\vartheta_\eps=\psi_j+O(\eps)$; also $a_\eps(\vartheta_\eps)=O(\eps)$. So the fixed point lies in chart A, where the return is the smooth fixed-order map $\mathcal P_\eps$ and $\vartheta=\psi$, and it is a fixed point of $\mathcal P_\eps$ within $O(\eps)$ of $(\psi_j,0)$. Theorem~\ref{thm:reciprocallocal} provides a fixed point of $\mathcal P_\eps$ which is unique in a fixed neighbourhood of $(\psi_j,0)$. For small $\eps$ both lie in that neighbourhood, so they coincide: the fixed point is the one-to-one locked state of Theorem~\ref{thm:reciprocallocal}, with the multipliers of Corollary~\ref{cor:period} and the stated stability. A periodic point of the full return remaining in the attracting tube must lie on the invariant graph: iterating its period in the normal-contraction inequality forces its distance to the graph to vanish. Hence there are no other periodic returns in that tube. Metric normal attraction makes the downward fixed points orbitally asymptotically stable for the full impulsive system.
\end{proof}

\paragraph{A geometric sufficient condition for regularity.}
We now establish the hypothesis of Theorem~\ref{thm:globalreciprocal} for a concrete class of impulses.

\begin{proposition}[Invariant emission sections imply regularity]\label{prop:sufficientregularity}
Suppose that the $C^2$ local impulse diffeomorphisms satisfy
\begin{equation}\label{eq:sufficientregularity}
 K_i^\eps(\Sigma_i)\subseteq\Sigma_i,
 \qquad i=1,2,
\end{equation}
for all $|\eps|$ sufficiently small, locally near $p_i$. Then the paired-event transition, with emissions recorded as in Definition~\ref{def:singleemission}, is regular in the sense of Definition~\ref{def:pairedregular}. In particular, the regularity assumption of Theorem~\ref{thm:globalreciprocal} follows from~\eqref{eq:sufficientregularity}.
\end{proposition}

\begin{proof}
Near $\Sigma_i$ introduce $C^2$ coordinates $(v,y)$ with $v=g_i(x)$, so that $\Sigma_i=\{v=0\}$. Let $k_i^\eps(v,y)$ be the normal coordinate of $K_i^\eps$. Exact invariance means $k_i^\eps(0,y)=0$. Hence the fundamental theorem of calculus gives
\[
 k_i^\eps(v,y)
 =v\int_0^1\partial_v k_i^\eps(tv,y)\,\dd t
 =:\alpha_i^\eps(v,y)v.
\]
Since $K_i^0=\id$, we have $\alpha_i^0=1$. By continuity, on a smaller fixed neighbourhood and for small $|\eps|$, $\alpha_i^\eps\ge1/2$. Equivalently,
\begin{equation}\label{eq:sectionfactor}
 g_i(K_i^\eps(x))=\alpha_i^\eps(x)g_i(x),
 \qquad\alpha_i^\eps(x)>0.
\end{equation}
Thus each kick preserves the emission section and its incoming and outgoing sides. In particular, a receiver which has not yet emitted cannot be carried across its emission section by a kick, and a receiver which has just emitted cannot be kicked back across that emission section.

Use the incoming section of Lemma~\ref{lem:pairedepisode}, and write $u=t_2-t_1$. Write $z=(u,w)$, with $u$ as the normal coordinate to the switching hypersurface and $w$ coordinates along it. Suppose $u\ge0$. The signed time $\widehat s_{12}$ of Lemma~\ref{lem:pairedepisode} is the unique solution of
\[
 g_2\bigl(\phi^2_{\widehat s_{12}}K_2^\eps\phi^2_{t_1(z)}x_2\bigr)=0.
\]
It is jointly $C^2$ by the implicit function theorem, since the derivative with respect to $\widehat s_{12}$ is $\dd g_2(F_2)$, uniformly bounded away from zero. On $u=0$, the receiver is already on $\Sigma_2$ before the kick and remains there by~\eqref{eq:sectionfactor}; hence $\widehat s_{12}(0,w,\eps)=0$. At $\eps=0$, $\widehat s_{12}(u,w,0)=u$. Integrating the derivative in $u$ gives
\[
 \widehat s_{12}(u,w,\eps)
 =u\int_0^1\partial_u \widehat s_{12}(tu,w,\eps)\,\dd t
 =u(1+O(|\eps|)),
\]
uniformly up to $u=0$. This is positive for $u>0$, so $s_{12}=\widehat s_{12}$ there. Similarly, on the other side, $s_{21}=\widehat s_{21}=-u(1+O(|\eps|))$. Thus no immediate impulse-induced emission is needed in this section-preserving case, while the natural emissions still occur.

Let $\Phi_t=(\phi_t^1,\phi_t^2)$ and define the product-space kicks
\[
 J_1^\eps(x_1,x_2)=(K_1^\eps x_1,x_2),
 \qquad
 J_2^\eps(x_1,x_2)=(x_1,K_2^\eps x_2).
\]
Immediately after the paired emissions the two event-order formulas are
\begin{align*}
 Y_{12}^\eps(z)&=J_1^\eps\circ\Phi_{s_{12}(z,\eps)}
 \circ J_2^\eps\circ\Phi_{t_1(z)}(z),\qquad u\ge0,\\
 Y_{21}^\eps(z)&=J_2^\eps\circ\Phi_{s_{21}(z,\eps)}
 \circ J_1^\eps\circ\Phi_{t_2(z)}(z),\qquad u\le0.
\end{align*}
Each formula extends jointly $C^2$ across $u=0$, with signed intermediate times. On that boundary $t_1=t_2$ and $s_{12}=s_{21}=0$, so the values coincide \emph{exactly} because $J_1^\eps J_2^\eps=J_2^\eps J_1^\eps$. Compose both formulas with the smooth free hit to the same fixed outgoing section, obtaining two $C^2$ branches, $E_\eps^+$ from $Y_{12}^\eps$ and $E_\eps^-$ from $Y_{21}^\eps$, with equal boundary values.

At $\eps=0$ the kicks are identities and both extended branches reduce, on their entire domains, to the same free transition $E_0$. Uniform $C^2$ regularity on compact closures and the parameter mean-value theorem imply
\[
 \|E_\eps^\pm-E_0\|_{C^1}\le C|\eps|.
\]
Straighten $u=0$ and restrict to a convex box. A line segment crosses that hyperplane at most once. Integrating the derivative estimate on the one or two pieces, using the equality of boundary values, gives
\[
 \|E_\eps-E_0\|_\infty+\Lip(E_\eps-E_0)\le C|\eps|.
\]
Each branch is $C^2$ up to $u=0$, so $E_\eps$ is piecewise $C^1$, and the last estimate is~\eqref{eq:pairedReg}. This is Definition~\ref{def:pairedregular}.
\end{proof}

\begin{corollary}[Emission sections aligned with impulses]\label{cor:alignedsections}
The paired-event transition is regular if, for $i=1,2$, either $K_i^\eps(x)=x+\eps \xi_i(x)d_i$, with $\xi_i$ a $C^2$ scalar function and $d_i$ a fixed nonzero direction, and $\Sigma_i$ is a short straight emission section parallel to $d_i$ and transverse to $F_i(p_i)$; or $K_i^\eps$ is the time-$\eps$ flow of a nonvanishing perturbative field $B_i$ and $\Sigma_i$ is an integral curve of $B_i$ transverse to $F_i$.
\end{corollary}

\begin{proof}
In the first case the affine defining function of $\Sigma_i$ satisfies $g_i(K_i^\eps x)=g_i(x)$; in the second, $\Sigma_i$ is locally invariant under the flow of $B_i$. In both cases~\eqref{eq:sufficientregularity} holds locally near $p_i$, and Proposition~\ref{prop:sufficientregularity} applies.
\end{proof}

Mere first-order tangency of $B_i$ to $\Sigma_i$ does not imply invariance under an arbitrary kick map. Invariance of the emission section is distinct from invariance of an isochron and does not force $H(0)=0$.

\begin{remark}[Why the general regularity hypothesis is genuine]\label{rem:backwardkick}
The single-emission convention alone does not imply Definition~\ref{def:pairedregular}. An impulse moving its receiver back across a section may produce a short recrossing; on opposite sides of switching, the reciprocal impulse can then be applied before or after a free flight of size $O(|\eps|)$. The two one-sided limits at the switching set may then differ by an amount of order $\eps^2$ that is nonzero for each fixed $\eps\neq0$; Example~A of Section~\ref{sec:analytic} exhibits exactly this mechanism. The general theorem therefore assumes regularity, while Proposition~\ref{prop:sufficientregularity} supplies one concrete class in which it is proved.
\end{remark}

\section{Small period mismatch and structural stability}\label{sec:nearbyperiods}

This section answers two questions. How does a small difference of the natural periods enter the phase balance as the coupling tends to zero? And, for a fixed small coupling strength, which perturbations of the system preserve the global dynamics? The second question does not settle the first: at an equal-period locked state $\lambda_{\rm rel}=1+\eps H'(\psi_*)+o(\eps)$, so the relative-phase multiplier lies within $O(\eps)$ of the neutral value~$1$. We therefore keep the coupling strength and the period perturbation as independent parameters.

\subsection{Period mismatch of the order of the coupling}
Let $\delta$ be a second small parameter, and consider a jointly $C^2$ family
\[
F_i^\delta,\qquad \Sigma_i^\delta,\qquad K_{i,\delta}^\eps,
\qquad i=1,2,
\]
which at $\delta=0$ is the equal-period system of Sections~\ref{sec:setting}--\ref{sec:global}, with $K_{i,\delta}^0=\id$. By hyperbolicity, $F_i^\delta$ has an attracting cycle $\Gamma_i^\delta$ near $\Gamma_i$, of period $T_i^\delta$ and angular frequency $\omega_i^\delta=2\pi/T_i^\delta$. For small $\delta$ the section $\Sigma_i^\delta$ is transverse to $\Gamma_i^\delta$ and meets it at a single point near $p_i$, which we take as phase origin.

We assume the following \emph{uniform phase hypotheses}. There are annular neighbourhoods $U_i^*$, independent of small $\delta$, on which the asymptotic phases $\Theta_i^\delta$ are $C^2$ and
\begin{equation}\label{eq:phaseuniformdelta}
	\sup_{|\delta|\le\delta_0}\|\Theta_i^\delta\|_{C^2(U_i^*)}<\infty,
	\qquad
	\Theta_i^\delta\longrightarrow\Theta_i\quad\hbox{in }C^2(U_i^*)
	\quad(\delta\to0),
\end{equation}
and the first-order phase responses $h_i^\delta$ of Definition~\ref{def:response}, computed for the $\delta$-system with phase origins at $\Sigma_i^\delta\cap\Gamma_i^\delta$, satisfy
\begin{equation}\label{eq:huniformdelta}
	\|h_i^\delta-h_i\|_{C^0}=O(|\delta|),
	\qquad
	\|h_i^\delta-h_i\|_{C^1}=o(1).
\end{equation}

On $\Gamma_1^\delta\times\Gamma_2^\delta$ the free flow changes the phase difference, during one natural period of oscillator~1, by
\[
(\omega_2^\delta-\omega_1^\delta)T_1^\delta=\Delta_\delta,
\]
the quantity~\eqref{eq:Deltadelta}. Thus $\Delta_0=0$ and $\Delta_\delta=O(|\delta|)$. Off the product of the cycles the return time differs from $T_1^\delta$ by $O(\|a\|)$, which adds a term $O(|\delta|\,\|a\|)$ to the drift.

\begin{proposition}[Fixed-order return with period mismatch]\label{prop:mismatchreturn}
	Assume Hypotheses~\ref{hyp:cycles} and~\ref{hyp:equalperiods} at $\delta=0$ and the uniform phase hypotheses~\eqref{eq:phaseuniformdelta}--\eqref{eq:huniformdelta}. Fix $M>0$ and a compact interval $J\subset(0,2\pi)$. Then, uniformly for sufficiently small $\eps>0$ and $|\delta|\le M\eps$, the two transverse fixed-point equations of the fixed-order return have a unique solution $a=a_{\eps,\delta}^{\rm fix}(\psi)$ on $J$, with $\|a_{\eps,\delta}^{\rm fix}\|_{C^1(J)}=O(\eps)$, and the remaining scalar fixed-point equation is
	\begin{equation}\label{eq:mismatchscalar}
		0=\Delta_\delta+\eps H(\psi)+R_{\eps,\delta}(\psi),
		\qquad
		\|R_{\eps,\delta}\|_{C^0(J)}=O(\eps^2),
		\qquad
		\|R_{\eps,\delta}\|_{C^1(J)}=o(\eps).
	\end{equation}
\end{proposition}

\begin{proof}
	The argument of Lemma~\ref{lem:phaseidentity} applies to the perturbed phases $\Theta_i^\delta$, with one difference: a free flight of duration $t$ now changes the phase difference by $(\omega_2^\delta-\omega_1^\delta)t$. The phase displacement of the return is therefore the free drift plus the two impulsive phase changes of~\eqref{eq:exactphase}. As computed above, the drift is $\Delta_\delta+O(|\delta|\,\|a\|)$ in the tubular coordinates. At $\eps=0$ the transverse derivative at $a=0$ is $D+O(|\delta|)$, uniformly on $J$, where $D=\diag(\rho_1,\rho_2)$; hence the fixed-point linearisation $I-D+O(|\delta|)$ is uniformly invertible. The joint $C^2$ dependence of the vector fields, sections and impulse maps, together with~\eqref{eq:phaseuniformdelta}, makes the fixed-order estimates for hitting times and coordinates uniform in $\delta$. By~\eqref{eq:huniformdelta}, replacing $h_i^\delta$ by $h_i$ costs $O(\eps|\delta|)$ in $C^0$ and $o(\eps)$ in $C^1$. Since $|\delta|\le M\eps$ and the transverse graph has size $O(\eps)$, all the additional terms are $O(\eps^2)$ in $C^0$. The transverse implicit-function argument in the proof of Theorem~\ref{thm:reciprocallocal}, applied to the expansion of Proposition~\ref{prop:returnexp}, now gives~\eqref{eq:mismatchscalar}.
\end{proof}

\begin{proof}[Proof of Theorem~\ref{thm:smallmismatch}]
	Dividing~\eqref{eq:mismatchscalar} by $\eps$ along~\eqref{eq:nudrift} gives the equation $H_\nu+o_{C^1}(1)=0$, and the monotonicity and intermediate-value argument in the proof of Theorem~\ref{thm:reciprocallocal} gives its unique zero near $\psi_*$. At this fixed point the return derivative retains the block form~\eqref{eq:blockderivative}, with transverse block $D+o(1)$ since $|\delta|\le M\eps$. The free drift contributes $(\omega_2^\delta-\omega_1^\delta)\tau$ over a return of duration $\tau$; since $\omega_2^\delta-\omega_1^\delta=O(\eps)$ and $\tau=T_1^\delta+O(\|a\|+\eps)$ in $C^1$, its derivative is $O(\eps^2)$ in $\psi$ and $O(\eps)$ in $a$. The Schur-complement argument leading to~\eqref{eq:relmultproof} then gives $\lambda_{\rm rel}$, and the other two multipliers converge to $\rho_1,\rho_2\in(0,1)$ because the transverse block tends to $D$. This gives the stability statement.
\end{proof}

\begin{corollary}[Generic reciprocal locking with small period mismatch]\label{cor:genericmismatch}
Consider a jointly $C^2$ family satisfying the phase hypotheses
\eqref{eq:phaseuniformdelta}--\eqref{eq:huniformdelta} and whose equal-period
system at $\delta=0$ has the translation impulses~\eqref{eq:openreciprocal}.
Fix a nonzero direction $d_1$. For every direction $d_2$ in the open,
dense, full-measure set of Theorem~\ref{thm:opencouplings}, there is
$\eta_0>0$ such that, for each fixed $0<\eta<\eta_0$ and each $M>0$,
there exist $\gamma,\eps_0>0$ with the following property: whenever
$0<\eps<\eps_0$, $|\delta|\le M\eps$, and
$|\Delta_\delta|\le \gamma\eps$, the full impulsive system possesses a
locally unique, orbitally asymptotically stable one-to-one locked state.
In particular, the conclusion allows $T_1^\delta\ne T_2^\delta$.
\end{corollary}

\begin{proof}
By the proof of Theorem~\ref{thm:opencouplings}, the equal-period balance
$H_\eta(\psi)=h_2(\psi)-\eta h_1(-\psi)$ has a simple downward zero $\psi_s\in(0,2\pi)$. Choose a compact
interval $I$ about $\psi_s$, contained in $(0,2\pi)$, and numbers
$\alpha,\gamma_*>0$ such that $H_\eta'\le-\alpha$ on $I$ and $H_\eta$ has opposite
signs at the endpoints, each of absolute value at least $2\gamma_*$.
Proposition~\ref{prop:mismatchreturn}, uniformly for
$|\delta|\le M\eps$, reduces the fixed-point equation on $I$ to
$H_\eta(\psi)+\Delta_\delta/\eps+R_{\eps,\delta}(\psi)/\eps=0$,
where $R_{\eps,\delta}/\eps\to0$ in $C^1(I)$ uniformly in $\delta$.
Taking $0<\gamma<\gamma_*$ and then $\eps_0$ sufficiently small, this function
is strictly decreasing and has opposite signs at the endpoints of $I$.
Its unique zero gives a locally unique locked state. The transverse
multipliers tend to $\rho_1,\rho_2\in(0,1)$. Writing
$\psi_{\eps,\delta}\in I$ for the resulting fixed-point phase, the
relative multiplier satisfies
$\lambda_{\rm rel}=1+\eps H_\eta'(\psi_{\eps,\delta})+o(\eps)
\le 1-\alpha\eps+o(\eps)<1$ uniformly for the stated parameters, by the
same Schur-complement argument as in Theorem~\ref{thm:smallmismatch}.
Hence the locked state is orbitally asymptotically stable. No paired-event regularity assumption is needed,
since $I$ stays away from simultaneous emission.
\end{proof}

\begin{corollary}[Locked period with mismatch]\label{cor:mismatchperiod}
	In the setting of Theorem~\ref{thm:smallmismatch}, suppose that $\Delta_{\delta(\eps)}=\eps\nu+O(\eps^2)$. Then $\psi_\eps=\psi_*+O(\eps)$, and the period of the locked state is
	\begin{equation}\label{eq:mismatchperiod}
		\mathcal T_{\eps,\delta}
		=T_1^\delta-\frac{\eps}{\omega_1^\delta}h_1^\delta(-\psi_*)+O(\eps^2)
		=T_2^\delta-\frac{\eps}{\omega_2^\delta}h_2^\delta(\psi_*)+O(\eps^2).
	\end{equation}
\end{corollary}

\begin{proof}
	With the $O(\eps^2)$ drift expansion and the $C^0$ bound in~\eqref{eq:mismatchscalar}, the zero equation in the proof of Theorem~\ref{thm:smallmismatch} gives $\psi_\eps=\psi_*+O(\eps)$. Over one period of the locked state, oscillator~1 accumulates the phase $\omega_1^\delta\mathcal T_{\eps,\delta}$ in free flight and receives one impulse, which changes its phase by $\eps h_1^\delta(-\psi_*)+O(\eps^2)$; the total is $2\pi$, which is the first expression in~\eqref{eq:mismatchperiod}. The same calculation for oscillator~2 gives the second, and the two agree to the stated order by $H_\nu(\psi_*)=0$ and the hypothesis $\Delta_{\delta(\eps)}=\eps\nu+O(\eps^2)$.
\end{proof}

\begin{theorem}[Invariant circle with period mismatch]\label{thm:mismatchglobal}
	In the setting of Proposition~\ref{prop:mismatchreturn}, suppose that the single-emission convention holds with the same $\tau_*$ for all $|\delta|\le M\eps$ and that the paired-event transition is regular (Definition~\ref{def:pairedregular}), with constants uniform in $\delta$. Then, for every sufficiently small $\eps>0$ and $|\delta|\le M\eps$, the return has a Lipschitz invariant circle $\mathcal C_{\eps,\delta}$, normally attracting in the metric sense, and in the base coordinate $\vartheta$ of the perturbed return cylinder the induced circle map has a lift
	\begin{equation}\label{eq:mismatchglobal}
		\widetilde q_{\eps,\delta}(\vartheta)
		=\vartheta+\Delta_\delta+\eps H(\vartheta)+r_{\eps,\delta}(\vartheta),
		\qquad
		\|r_{\eps,\delta}\|_\infty=O(\eps^2).
	\end{equation}
\end{theorem}

\begin{proof}
	We repeat the construction of Section~\ref{sec:reciprocalglobal} for the perturbed system. Since the return cylinder is built from two observation sections rather than from the single section $\mathcal S_1$, the mismatch raises no new question about the number of emissions recorded in one return. The gluing map $\sigma_{\eps,\delta}$ is $C^2$ jointly in $(x,\eps,\delta)$, and for $|\delta|\le M\eps$ the free flows differ from the equal-period ones by $O(|\delta|)=O(\eps)$ in $C^1$. Hence Lemmas~\ref{lem:gluing} and~\ref{lem:returncylinder} hold with constants uniform in $\delta$, the only change being that the angular component of $X_B^0$ now differs from $\psi$ by $O(\eps)$; with the uniform regularity assumed, the graph transform of Theorem~\ref{thm:globalcircle} applies with constants independent of $|\delta|\le M\eps$.

More precisely, in the common cylinder coordinates write $\mathcal R_{\eps,\delta}-\mathcal R_{0,0}=(\mathcal R_{\eps,\delta}-\mathcal R_{0,\delta})+(\mathcal R_{0,\delta}-\mathcal R_{0,0})$. The first difference is controlled by the assumed uniform paired-event regularity and the fixed-order estimates, the second by the smooth dependence of the uncoupled flights on $\delta$; hence
\[
\|\mathcal R_{\eps,\delta}-\mathcal R_{0,0}\|_\infty
+\Lip(\mathcal R_{\eps,\delta}-\mathcal R_{0,0})
\le C\bigl(|\eps|+|\delta|\bigr)
\le C_M|\eps|,\qquad |\delta|\le M|\eps|.
\]
Here the constants and the product tube are independent of the admissible parameters. This is the uniform estimate needed for the graph transform and its normal-contraction bounds.

The phase calculation of Lemma~\ref{lem:globalphaseexpansion} must be modified in one point: free flight no longer preserves the phase difference, but changes it by $(\omega_2^\delta-\omega_1^\delta)t$ in time $t$. Over one return, whose duration is $T_1^\delta+O(\|a\|+|\eps|)$, this adds $\Delta_\delta+O(|\delta|(\|a\|+|\eps|))=\Delta_\delta+O(\eps^2)$ on graphs of size $O(\eps)$, as in the proof of Proposition~\ref{prop:mismatchreturn}. The two impulsive phase changes are unchanged to first order, by~\eqref{eq:huniformdelta}. The change of coordinates on chart B still contributes only $O(\eps^2)$, as in~\eqref{eq:chartOeps2}, because the return, drift included, still moves points by $O(\eps)$. This gives~\eqref{eq:mismatchglobal}.

\end{proof}

\begin{corollary}[Global dynamics with period mismatch]\label{cor:mismatchMS}
	In the setting of Theorem~\ref{thm:mismatchglobal}, let $\delta(\eps)$ be a path with $|\delta(\eps)|\le M\eps$ satisfying~\eqref{eq:nudrift}, and let $H_\nu=\nu+H$. If $H_\nu(0)\neq0$ and $H_\nu$ has at least one zero, all of them simple, then for every sufficiently small $\eps>0$ the circle map $q_{\eps,\delta(\eps)}$ is topologically Morse--Smale, with exactly one locked state near each zero of $H_\nu$, attracting at downward and unstable in the relative-phase direction at upward zeros, and none near simultaneous emission.
\end{corollary}

\begin{proof}
	Along~\eqref{eq:nudrift},
	\[
	\widetilde q_{\eps,\delta(\eps)}(\vartheta)
	=\vartheta+\eps H_\nu(\vartheta)+\widehat r_\eps(\vartheta),
	\qquad
	\|\widehat r_\eps\|_\infty=o(\eps).
	\]
	Since $H_\nu(0)\neq0$ and the zero set of $H_\nu$ is finite (Remark~\ref{rem:finitezeros}), there is $\kappa\in(0,\pi/8)$ such that all zeros of $H_\nu$ lie in the interior of $J_\nu=[2\kappa,2\pi-2\kappa]$. As in the proof of Theorem~\ref{thm:globalreciprocal}, this $\kappa$ is fixed first, and the return cylinder in the proof of Theorem~\ref{thm:mismatchglobal} is constructed with it; the remaining constants are chosen afterwards. Then $J_\nu$ lies in chart A, where $\vartheta$ is the perturbed phase difference $\psi^\delta:=\Theta_2^\delta-\Theta_1^\delta$, and the fixed-order estimates, together with~\eqref{eq:phaseuniformdelta}--\eqref{eq:huniformdelta}, give $\Lip(\widehat r_\eps|_{J_\nu})=o(\eps)$. Simultaneous emission stays within $O(\eps)$ of $\vartheta=0$, where $H_\nu\neq0$. Lemma~\ref{lem:circle}, applied to $H_\nu$, gives the conclusion.
\end{proof}

\begin{corollary}[Loss of locking]\label{cor:mismatchloss}
	In the setting of Theorem~\ref{thm:mismatchglobal}, let $\delta(\eps)$ satisfy $|\delta(\eps)|\le M\eps$ and~\eqref{eq:nudrift}, and put $H_\nu=\nu+H$. Suppose that $H_\nu$ has no zero. Then, for every sufficiently small $\eps>0$, the system with parameters $(\eps,\delta(\eps))$ has no locked state in the product tube of Theorem~\ref{thm:mismatchglobal}.
\end{corollary}

\begin{proof}
	A locked state in the tube gives a fixed point of the return there. By metric normal attraction it lies on the invariant circle $\mathcal C_{\eps,\delta(\eps)}$,  so it is a fixed point of $q_{\eps,\delta(\eps)}$. Along~\eqref{eq:nudrift}, however, $\widetilde q_{\eps,\delta(\eps)}(\vartheta)-\vartheta=\eps H_\nu(\vartheta)+o(\eps)$ uniformly, as in the proof of Corollary~\ref{cor:mismatchMS}, and $|H_\nu|$ has a positive minimum on $\Sone$; so for small $\eps>0$ the displacement neither vanishes nor reaches $\pm2\pi$.
\end{proof}

\begin{remark}[Scale of the mismatch]\label{rem:mismatchscale}
	For each fixed sufficiently small $\eps>0$, an equal-period locked state persists under sufficiently small admissible changes of the natural periods. As $\eps\to0$, however, a free phase drift of order $\eps$ competes at leading order with the impulsive phase correction, and the balance becomes $\nu+H=0$.
\end{remark}

\subsection{Structural stability at fixed coupling}
Structural stability requires the switching construction to be continuous on a neighbourhood of systems, not at a single system. Fix a class of admissible $C^2$ reciprocal systems whose paired-event returns are regular, and let $\mathcal U$ be a relative $C^2$ neighbourhood of the reference system in that class. The admissible class need not be open in the unrestricted space of all impulse maps: exact section preservation, for example, is not preserved by arbitrary $C^2$ perturbations. Assume that, for every system in $\mathcal U$, the emission sections are transverse, the single-emission convention holds with the same $\tau_*$, the return cylinder of Section~\ref{sec:reciprocalglobal} is defined, and its local returns vary continuously with the system in Lipschitz norm, with uniform bi-Lipschitz bounds. The gluing map is a composition of flow, hitting-time and impulse maps in a fixed order, so it varies continuously with the system in the $C^1$ norm: if the systems converge in $C^2$, their gluing maps converge in $C^1$.

\begin{theorem}[Structural stability of reciprocal circle locking]\label{thm:reciprocalstruct}
	Fix a sufficiently small $\eps>0$ and an equal-period reference system with such a neighbourhood $\mathcal U$. Assume that $H(0)\neq0$ and that $H$ has at least one zero, all of its zeros being simple. Then, after shrinking $\mathcal U$, every system in $\mathcal U$ has a nearby Lipschitz invariant circle, normally attracting in the metric sense. The induced circle homeomorphism has the same number and cyclic order of attracting and repelling locked states as the reference one, and is topologically conjugate to it. An admissible perturbation may change the vector fields, the impulse maps, the emission sections and the natural periods, provided the stated uniform return regularity is retained.
\end{theorem}

\begin{proof}
The continuity of the local returns in Lipschitz norm and of the gluing map in the $C^1$ norm, with the uniform bounds, keeps the graph transform a uniform contraction, so the invariant circle persists. Let $\mathcal R$ be the return of a perturbed system, $\mathcal C$ its invariant circle, and $z$ a fixed point of $\mathcal R$ in the common tube. Metric normal attraction gives
	\[
	\operatorname{dist}(z,\mathcal C)
	=\operatorname{dist}(\mathcal R^n z,\mathcal C)
	\le C_*\lambda^n\operatorname{dist}(z,\mathcal C)
	\]
	for every $n$, hence $z\in\mathcal C$: every nearby fixed point is a fixed point of the circle map.
	
	For the reference system, the displacement has constant sign near simultaneous emission, because $H(0)\neq0$. Each locked state lies away from simultaneous emission and is a hyperbolic fixed point of a smooth fixed-order return in dimension three. After shrinking $\mathcal U$, the hyperbolic continuation theorem therefore gives exactly one fixed point of the full return near each reference locked state, with the same stability type, and by the preceding paragraph it lies on the perturbed invariant circle.
	
	It remains to exclude other fixed points. The graph transforms depend continuously on the system as operators in the sup norm, and so do their fixed graphs; hence the displacements of the circle maps vary continuously in $C^0$. On a fixed-order interval around a zero of $H$, the two Lipschitz invariant graphs $a$ and $\widehat a$ do not spoil monotonicity: since $\Lip_a(Q)=O(\eps)$ and both graphs have Lipschitz constant $O(\eps)$, their contribution to the Lipschitz constant of the difference of the displacements is at most
	\[
	\Lip_a(Q)\bigl(\Lip(\widehat a)+\Lip(a)\bigr)=O(\eps^2),
	\]
	well below the slope $c\eps$ of the reference displacement at the fixed small coupling. Outside the chosen neighbourhoods of the fixed points, the reference displacement is bounded away from zero, and $C^0$ closeness excludes new fixed points. The number, type and cyclic order of the fixed points are therefore unchanged, and Lemma~\ref{lem:circleconjugacy} gives the conjugacy.
\end{proof}

\begin{theorem}[Structural stability of stroboscopic locking]\label{thm:nearbyperiods}
	Fix a sufficiently small coupling strength $\eps>0$, and suppose that the stroboscopic map $P_\eps$ of Theorem~\ref{thm:strobe} has a normally attracting invariant circle on which the induced circle diffeomorphism is Morse--Smale, with a nonempty finite set of hyperbolic fixed points. Then every sufficiently small $C^1$ perturbation of $P_\eps$ has a nearby normally attracting invariant circle, and the induced circle diffeomorphism is topologically conjugate to the original one.
\end{theorem}

\begin{proof}
	Normal hyperbolicity gives the perturbed invariant circle and its stable foliation. The induced circle diffeomorphism is $C^1$-close to the original one, and the structural stability of Morse--Smale circle diffeomorphisms gives the conjugacy.
\end{proof}

\section{Sharpness of the hypotheses and a numerical example}\label{sec:examples}

This section treats the degenerate balance $H(0)=0$, gives three analytic examples that delimit the regularity hypothesis of Theorem~\ref{thm:globalreciprocal}, and illustrates the period-mismatch results numerically.

\subsection{Exceptional balance at simultaneous emission}\label{sec:exceptional}

Theorem~\ref{thm:globalreciprocal} assumes $H(0)\neq0$. Here we record what changes when
\[
H(0)=h_2(0)-h_1(0)=0,
\]
that is, when at simultaneous emission the two impulses change the two asymptotic phases by the same amount to first order. Neither phase response need vanish, and there are two cases.

If $h_1(0)=h_2(0)\neq0$, both oscillators are advanced, or both delayed, by the same amount: to first order their common phase changes and their phase difference does not. A locked state may exist near $\psi=0$, but its position and stability depend on the paired-event return on each side of the switching set. The first-order balance no longer decides, because its zero lies exactly where the order of the emissions changes.

If $h_1(0)=h_2(0)=0$, each impulse field is tangent to the isochron at the emission point, and the two first-order phase shifts vanish separately. Since $H'(0)=h_2'(0)+h_1'(0)$, a nonzero $H'(0)$ still makes the first-order displacement restoring or destabilising on the two smooth sides. Whether there is a fixed point at $\psi=0$, or near it, can depend on the second-order terms of the paired-event transition.

In either case, suppose that near $\psi=0$ the displacement of the circle map is given on the two sides of the switching point by
\[
D_\eps^\pm(\psi)
=\eps H(\psi)+r_\eps^\pm(\psi),
\qquad
\|r_\eps^\pm\|_{C^0}=O(\eps^2).
\]
If the paired-event transition is regular, so that the invariant circle of Theorem~\ref{thm:globalcircle} exists, and the displacement has a zero $\psi_\eps=O(\eps)$ at which it changes sign from positive to negative, the corresponding nearly simultaneous locked state is attracting on the circle; if the sign changes from negative to positive, it is repelling. In the repelling case every nearby orbit other than the fixed point leaves the neighbourhood of simultaneous emission and enters a region where the order of the emissions is fixed and the local theory applies. The exceptional case therefore affects the dynamics only near simultaneous emission.

The condition $H(0)=0$ has codimension one unless a symmetry imposes it, as happens for two identical oscillators with identical impulses, for which $H(\psi)=h(\psi)-h(-\psi)$ is odd. It is excluded by the hypothesis $H(0)\neq0$ of Theorem~\ref{thm:globalreciprocal}.

\subsection{Three analytic examples}\label{sec:analytic}
\noindent\textbf{A. Failure of continuity without section preservation.}
On annuli about $r_i=1$ let $\dot r_i=-\lambda_i(r_i-1)$, $\dot\theta_i=1$ ($\lambda_i>0$), with emission isochrons $\Sigma_i=\{\theta_i=0\}$ and impulses
\[
K_1^\eps(r,\theta)=(r,\theta+\eps(a+b\sin\theta)),\qquad
K_2^\eps(r,\theta)=(r,\theta+c\eps),
\qquad a<0,\ c<0,\ bc\neq0;
\]
the signs of $a$ and $c$ ensure that no kick carries a receiver that has not yet emitted forward across its section, while a receiver that has just emitted may be carried back across it, the resulting recrossing being ignored (Definition~\ref{def:singleemission}). Then $H(\psi)=c-a+b\sin\psi$, which has $H(0)\neq0$ and exactly two simple zeros if $0<|c-a|<|b|$. With one recorded emission per oscillator per revolution, approach natural simultaneity from initial phases $(0,-d)$ and $(-d,0)$, $d\downarrow0$. At the common observation time $t=\pi$ the two post-episode limits are $L_{12}=(\pi+\eps(a+b\sin(-c\eps)),\pi+c\eps)$ and $L_{21}=(\pi+a\eps,\pi+c\eps)$, so
\[
L_{12}-L_{21}=\bigl(\eps b\sin(-c\eps),0\bigr)=\bigl(-bc\eps^2+O(\eps^4),0\bigr)\neq0
\]
for all sufficiently small $\eps>0$. For $(a,b,c)=(-1,3,-2)$ this gives $H(\psi)=-1+3\sin\psi$, $H(0)=-1$, $L_{12}=(\pi-\eps+3\eps\sin(2\eps),\pi-2\eps)$ and $L_{21}=(\pi-\eps,\pi-2\eps)$. Thus the paired-event transition is discontinuous, even with isochronal emission, on an open set of triples $(a,b,c)$; in particular, condition~\eqref{eq:pairedReg} is not generic within this family. This disproves automatic global regularity, not every possible global locking conclusion.

\medskip\noindent\textbf{B. Regular return with aligned emission sections.}
Let $\dot r_i=-\lambda_i(r_i-1)$, $\dot\theta_i=1+c_i(r_i-1)$ ($\lambda_i,c_i>0$). Their asymptotic phases are $\Theta_i=\theta_i+(c_i/\lambda_i)(r_i-1)$. Choose $\Sigma_i=\{\theta_i=0\}$ and
\[
K_i^\eps(r,\theta)=(r+\eps f_i(\theta),\theta),\qquad
f_1=\lambda_1/c_1,\quad f_2=(\lambda_2/c_2)(\tfrac12+\sin\theta).
\]
The emission section is preserved exactly, since $\theta_i$ is unchanged. Nonetheless $h_1=1$, $h_2(\theta)=\tfrac12+\sin\theta$, and
\[
H(\psi)=\sin\psi-\tfrac12,\qquad H(0)=-\tfrac12,\qquad
\psi_u=\pi/6,\quad\psi_s=5\pi/6.
\]
Both zeros are simple. Proposition~\ref{prop:sufficientregularity} supplies the regularity required by Theorem~\ref{thm:globalreciprocal}, which therefore applies. In particular, alignment of an emission section does not entail $H(0)=0$.

\medskip\noindent\textbf{C. Regularity without preservation of emission sections.}
With the oscillators and sections of Example~A, fix $0<d<1/20$ and a smooth $\eta:\Sone\to[0,1]$ with $\eta(\theta)=0$ for $\operatorname{dist}(\theta,0)\le d$ and $\eta(\theta)=1$ for $\operatorname{dist}(\theta,0)\ge2d$, and let $K_i^\eps$ be the time-$\eps$ flows of
\[
B_1=2\,\partial_{\theta_1},\qquad
B_2=\bigl(\tfrac{11}{5}+\eta(\theta_2)\sin\theta_2\bigr)\,\partial_{\theta_2}.
\]
Both fields advance the phase and their flows are diffeomorphisms for either sign of $\eps$; on $|\theta|<d-3|\eps|$ they are the translations $\theta\mapsto\theta+2\eps$ and $\theta\mapsto\theta+\tfrac{11}{5}\eps$, so neither emission section is invariant for $\eps\neq0$.

Nevertheless, the paired-event transition is regular. Take the incoming and outgoing sections $\chi=\mp\sigma$, $\chi=(\theta_1+\theta_2)/2$, and restrict $\psi=\theta_2-\theta_1$ by $|\psi|+6|\eps|<d$. Each receiver is then kicked where the impulses are these translations; a forward kick across a section is recorded as an emission at that instant and a backward recrossing, possible for $\eps<0$, is ignored (Definition~\ref{def:singleemission}), so each oscillator receives exactly one kick. Since the translations commute with the free flow, the outgoing state is given by the same formula for both event orders and both signs of $\eps$: the episode lasts $t_\eps=2\sigma-\tfrac{21}{10}\eps$ and, with $a_i=r_i-1$,
\[
E_\eps(\psi,a_1,a_2)=\bigl(\psi+\tfrac{\eps}{5},\,e^{-\lambda_1t_\eps}a_1,\,e^{-\lambda_2t_\eps}a_2\bigr).
\]
This map is smooth and $e^{-\lambda_it_\eps}-e^{-2\sigma\lambda_i}=O(|\eps|)$, so Definition~\ref{def:pairedregular} holds on this domain, although neither condition in~\eqref{eq:sufficientregularity} does.

Finally $\Theta_i=\theta_i$, so $H(\psi)=\tfrac15+\eta(\psi)\sin\psi$ and $H(0)=\tfrac15$. Where $\eta<1$, $|\eta\sin\psi|<\sin 2d<\tfrac15$; since $\arcsin\tfrac15>2d$, $H$ has exactly two simple zeros, the restoring $\psi_s=\pi+\arcsin\tfrac15$ and the destabilising $\psi_u=2\pi-\arcsin\tfrac15$. Thus Theorem~\ref{thm:globalreciprocal} applies, and regularity of the paired-event transition is strictly more general than preservation of the emission sections.

\subsection{Numerical example with a small period mismatch}\label{sec:numerics}

We illustrate the theory on two oscillators of different kinds, coupled by translations. The coupling strength is $\eps$, an independent parameter $\delta$ changes the natural period of the second oscillator, and the computations follow the path $\delta=\eps$. No numerical value is used in the proofs.

\paragraph{Model.}
A radial cubic normal-form oscillator, with coordinates $(u_1,v_1)$, is coupled to a Van der Pol oscillator, with coordinates $(u_2,v_2)$:
\[
\begin{aligned}
	\dot u_1&=k(1-u_1^2-v_1^2)u_1-\omega_0v_1, &\qquad \dot u_2&=c_\delta v_2,\\
	\dot v_1&=\omega_0u_1+k(1-u_1^2-v_1^2)v_1, &\qquad \dot v_2&=c_\delta\bigl(\mu(1-u_2^2)v_2-u_2\bigr),
\end{aligned}
\]
with $k=2/5$, $\mu=3$ and $c_\delta=1+\delta\nu/(2\pi)$. Here $T_0\simeq8.8590954993$ is the period of the Van der Pol cycle for $c_\delta=1$, and $\omega_0=2\pi/T_0\simeq0.7092355317$, so that the radial oscillator also has period $T_0$. The factor $c_\delta$ changes only the time parametrisation of the Van der Pol field, and oscillator~1 does not depend on $\delta$. Hence $\Theta_i^\delta=\Theta_i$ and $h_i^\delta=h_i$, the uniform phase hypotheses~\eqref{eq:phaseuniformdelta}--\eqref{eq:huniformdelta} hold trivially, and
\[
\Delta_\delta=(c_\delta\omega_0-\omega_0)T_0=\delta\nu .
\]
We take $\nu=1/20$ and $\delta=\eps$, so that $\Delta_\delta=\eps\nu$ exactly.

\paragraph{Sections and impulses.}
Let
\[
\begin{aligned}
	\Sigma_1&=\{v_1=0,\ u_1>0\}\cap V_1,&\qquad
	K_1^\eps(u_1,v_1)&=(u_1-\beta\eps,v_1),\\
	\Sigma_2&=\{u_2=0,\ v_2>0\}\cap V_2,&\qquad
	K_2^\eps(u_2,v_2)&=(u_2,v_2-\eps),
\end{aligned}
\]
with $\beta=3/20$. These have the form of the reciprocal translations~\eqref{eq:openreciprocal}, with $\eta=\beta$, $d_1=-e_1$ and $d_2=-e_2$; the numerical value $\beta=3/20$ is not asserted to satisfy the unspecified small-ratio bound of Theorem~\ref{thm:opencouplings}. Each impulse leaves unchanged the function defining the section of its receiver, $g_1=v_1$ or $g_2=u_2$, so~\eqref{eq:sufficientregularity} holds, with $\alpha_i^\eps\equiv1$ in~\eqref{eq:sectionfactor}, and the paired-event transition is regular by Proposition~\ref{prop:sufficientregularity}, uniformly in $\delta$.

\paragraph{Phase balance.}
Since $\Theta_1=\arg(u_1+iv_1)$, we have $h_1(\theta)=\beta\sin\theta$. The phase gradient $Z_2=(Z_{2,u},Z_{2,v})$ of the Van der Pol cycle, computed from the periodic adjoint equation with the normalisation $Z_2\cdot F_2=\omega_0$, gives $h_2=-Z_{2,v}$, so that
\[
H_\nu(\psi)=\nu-Z_{2,v}(\psi)+\beta\sin\psi .
\]
The effective field is $W=-e_2+\omega_0^{-1}\beta\sin\psi\,F_2$, and it does not vanish: a zero would require $v_2=0$ and $\abs{u_2}\ge\omega_0/\beta>4$, whereas the rigorous Van der Pol amplitude bound $\abs{u_2}<2.3439<\omega_0/\beta$, valid for every $\mu\neq0$~\cite{Odani2000}, holds along the entire limit cycle. Deforming $\beta$ to $0$ keeps $W\neq0$, so $\wind_{\Gamma_2}(W)=0$ and $m=\pm1$ in~\eqref{eq:windingnumber}. The winding calculation applies to $H$, whereas the following numerically determined zeros belong to the shifted balance $H_\nu$ (Figure~\ref{fig:H}):
\[
\begin{aligned}
	\psi_s&\simeq2.8392984432,&\qquad H_\nu'(\psi_s)&\simeq-1.1593532404,\\
	\psi_u&\simeq5.8988225566,&\qquad H_\nu'(\psi_u)&\simeq\phantom{-}1.2647308033.
\end{aligned}
\]
The first is downward (restoring) and the second upward (destabilising). No additional zeros were detected numerically, and $H_\nu(0)\simeq0.347\neq0$. This is a numerical zero search, not a certified exhaustive count. The simple restoring zero illustrates the local conclusion of Theorem~\ref{thm:smallmismatch}; the global Morse--Smale conclusion for these particular numerical parameters additionally requires verification that all zeros of $H_\nu$ are simple.

\begin{figure}[htbp]
	\centering
	\includegraphics[width=.83\linewidth]{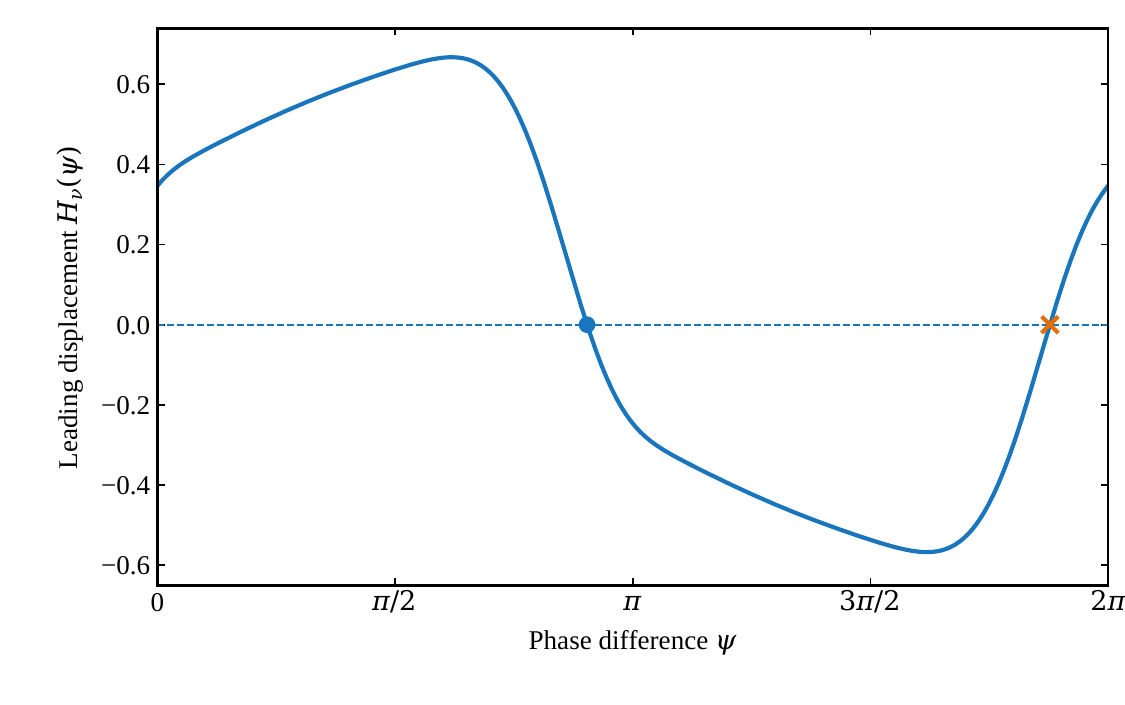}
	\caption{The leading phase displacement $H_\nu(\psi)=\nu-Z_{2,v}(\psi)+\beta\sin\psi$ of the example, with its restoring zero $\psi_s$ (filled circle) and destabilising zero $\psi_u$ (cross), both away from simultaneous emission, $\psi=0$.}
	\label{fig:H}
\end{figure}

\paragraph{Complete return.}
The Poincar\'e return is computed from the two planar equations, with detection of the crossings of the emission sections; no scalar phase recurrence is used. Table~\ref{tab:locking} lists its attracting fixed point for four values of $\eps$, together with the first-order predictions of Theorem~\ref{thm:smallmismatch} and Corollary~\ref{cor:mismatchperiod}. The period of the locked state differs from both natural periods (at $\eps=0.05$ these are $T_0$ and $T_0/c_\eps\simeq8.8556$), $\lambda_{\rm rel}<1$, and the errors of the predictions are proportional to $\eps^2$. Evaluating the phase component of the complete return on the solved transverse graph, as in Proposition~\ref{prop:mismatchreturn}, gives a displacement that differs from $\eps H_\nu(\psi)$ by a quantity that scales numerically as $\eps^2$, uniformly on compact subsets of $(0,2\pi)$. The integrations use relative tolerance $2\times10^{-10}$ and absolute tolerance $2\times10^{-12}$. The accompanying Python code computes the periodic orbit and adjoint, locates the physical emission events, evaluates the variational return, and compares the numerical outputs with Figure~\ref{fig:H} and Table~\ref{tab:locking} using explicit validation tolerances.

\begin{table}[htbp]
	\centering\small
	\caption{Fixed points of the complete return along $\delta=\eps$, with the first-order predictions $1+\eps H_\nu'(\psi_s)$ for the relative-phase multiplier and $T_0-\eps h_1(-\psi_s)/\omega_0$ for the period of the locked state.}
	\label{tab:locking}
	\begin{tabular}{@{}rrrrrr@{}}\toprule
		$\eps$&Computed phase&Computed $\lambda_{\rm rel}$&Predicted $\lambda_{\rm rel}$&Computed period&Predicted period\\\midrule
		0.0125&2.838102644&0.985616651&0.985508084&8.859886284&8.859882556\\
		0.0250&2.836908450&0.971447765&0.971016169&8.860684513&8.860669612\\
		0.0500&2.834524903&0.943737882&0.942032338&8.862303225&8.862243725\\
		0.1000&2.829777410&0.890725299&0.884064676&8.865629122&8.865391951\\\bottomrule
	\end{tabular}
\end{table}
\FloatBarrier

\section{Discussion}\label{sec:discussion}

Stable reciprocal synchronisation can occur between oscillators with
entirely different equations and cycle geometries.
Theorem~\ref{thm:opencouplings} establishes stable one-to-one locking
for any pair of planar hyperbolic attracting limit cycles with equal
natural periods, through arbitrarily weak reciprocal translation
impulses with generic directions and suitable positive strength ratios.
The resulting periodic motion attracts nearby trajectories of the
full impulsive system, including transverse perturbations.

The mechanism is topological: the phase gradient turns once around
each cycle, forcing a fixed-direction perturbation to produce both
phase advances and delays. The identity
$H=Z_2\cdot W$ in~\eqref{eq:effectiveidentity} expresses the reciprocal
phase balance through contacts between the effective impulse field
and the isochrons. The winding criterion converts this geometry into
lower bounds on the numbers of attracting and unstable locked states.
Thus the existence and multiplicity of synchronised rhythms are
linked to an integer invariant of the interaction geometry.

Theorem~\ref{thm:globalreciprocal} organises the locked states on a
normally attracting Lipschitz invariant circle with topologically
Morse--Smale dynamics. The two-chart return connects the different
emission orders, while the analytic examples in
Section~\ref{sec:analytic} reveal how the geometry of impulses and
emission sections determines the regularity of the global return.
This organisation persists under the admissible perturbations
considered in Section~\ref{sec:nearbyperiods}.

Locking does not require exactly equal natural periods: by
Corollary~\ref{cor:genericmismatch}, it persists over a nonzero range
of mismatches on the scale of the coupling.
The shifted balance $H_\nu=\nu+H$ describes the competition between
free phase drift and interaction-induced phase correction, governing
the adjustment of the locked phase and the loss of locking.
The numerical example in Section~\ref{sec:numerics} illustrates
this mechanism with a radial oscillator and a Van der Pol oscillator,
whose common locked period differs from both natural periods.

These results offer a topological perspective on the pervasiveness
of synchronisation described by Strogatz in
\emph{Sync}~\cite{Strogatz2003}.
Attraction stabilises the oscillations, phase geometry supplies
advances and delays, and reciprocal interaction can organise them
into a stable common rhythm. Oscillators can therefore synchronise
through geometric properties they share despite the differences
in their individual dynamics.

\section*{Declarations}
\addcontentsline{toc}{section}{Declarations}
\paragraph{Funding.}
Jorge Buescu acknowledges partial support from Funda\c{c}\~ao para a Ci\^encia e a Tecnologia through project UID/04561/2025 (\doi{10.54499/UID/04561/2025}). Henrique M. Oliveira acknowledges partial support from FCT/Portugal and the Recovery and Resilience Plan (PRR), through CAMGSD, under projects UID/04459/2025 (\doi{10.54499/UID/04459/2025}) and UID/PRR/04459/2025 (\doi{10.54499/UID/PRR/04459/2025}).

\paragraph{Competing interests.}
The authors declare no competing interests.

\paragraph{Data availability.}
All numerical data used in the figures are generated from the equations
and parameter values stated in Section~\ref{sec:numerics}.

\paragraph{Code availability.}
The complete reproducibility package for the numerical computations
described in Section~\ref{sec:numerics}, including the Python source code,
numerical data, validation outputs, and the files used to generate
Figure~5 and Table~1, is publicly available on Zenodo at
\href{https://doi.org/10.5281/zenodo.23299155}
{doi:10.5281/zenodo.23299155}.

\paragraph{Ethics and consent.}
The work involves no human participants, animal subjects or personal information. Ethics approval and consent to participate or publish are not applicable.

\paragraph{Author contributions.}
Jorge Buescu and Henrique M. Oliveira contributed to the conceptualisation, mathematical analysis and preparation of the manuscript. Henrique M. Oliveira developed the numerical implementation used for the illustrative computations.

\paragraph{AI-assisted preparation.}
AI tools were used during the revision of English text. The authors take full responsibility for the mathematical arguments, numerical results and final manuscript.

\setlength{\bibsep}{0pt}
\bibliographystyle{plain}
\bibliography{HuygensTopologicalTwoClocks}
\end{document}